\documentclass[a4paper]{amsart}
\usepackage{amssymb}
\usepackage[all]{xy}
\usepackage{hyperref,aliascnt}
\usepackage{enumerate}
\usepackage{stmaryrd}
\usepackage{mathtools}

\numberwithin{equation}{section}

\newtheorem{lma}{Lemma}[section]

\newaliascnt{thmCt}{lma}
\newtheorem{thm}[thmCt]{Theorem}
\aliascntresetthe{thmCt}

\newaliascnt{corCt}{lma}
\newtheorem{cor}[corCt]{Corollary}
\aliascntresetthe{corCt}

\newaliascnt{prpCt}{lma}
\newtheorem{prp}[prpCt]{Proposition}
\aliascntresetthe{prpCt}

\newcounter{theoremintro}

\newtheorem{thmIntro}[theoremintro]{Theorem}
\newtheorem{prpIntro}[theoremintro]{Proposition}
\theoremstyle{definition}

\newtheorem{pbmIntro}[theoremintro]{Problem}

\newaliascnt{pgrCt}{lma}
\newtheorem{pgr}[pgrCt]{}
\aliascntresetthe{pgrCt}

\newaliascnt{dfnCt}{lma}
\newtheorem{dfn}[dfnCt]{Definition}
\aliascntresetthe{dfnCt}

\newaliascnt{rmkCt}{lma}
\newtheorem{rmk}[rmkCt]{Remark}
\aliascntresetthe{rmkCt}

\newaliascnt{qstCt}{lma}
\newtheorem{qst}[qstCt]{Question}
\aliascntresetthe{qstCt}

\newaliascnt{pbmCt}{lma}
\newtheorem{pbm}[pbmCt]{Problem}
\aliascntresetthe{pbmCt}

\newaliascnt{exaCt}{lma}
\newtheorem{exa}[exaCt]{Example}
\aliascntresetthe{exaCt}

\newcommand{\dimnuc}{\dim_{\rm{nuc}}}
\newcommand{\thin}{{\rm{tb}}}
\newcommand{\NN}{\mathbb{N}}
\newcommand{\CC}{\mathbb{C}}
\newcommand{\KK}{\mathcal{K}}
\newcommand{\Bdd}{\mathcal{B}}
\newcommand{\ca}{$C^*$-al\-ge\-bra}
\newcommand{\axiomO}[1]{(O#1)}
\newcommand{\SubSep}{\mathrm{Sub}_{\mathrm{sep}}}
\newcommand{\soft}{{\rm{soft}}}
\newcommand{\alg}{\mathrm{alg}}
\newcommand{\CatCu}{\ensuremath{\mathrm{Cu}}}
\newcommand{\CatW}{\mathrm{W}}
\newcommand{\NNbar}{\overline{\mathbb{N}}}
\newcommand{\andSep}{\,\,\,\text{ and }\,\,\,}

\newcommand{\CuSgp}{$\CatCu$-sem\-i\-group}
\newcommand{\CuMor}{$\CatCu$-mor\-phism}
\newcommand{\ihom}[1]{\llbracket #1 \rrbracket}

\DeclareMathOperator{\LAff}{LAff}
\DeclareMathOperator{\Cu}{Cu}
\DeclareMathOperator{\Lsc}{Lsc}
\DeclareMathOperator{\id}{id}
\DeclareMathOperator{\locdim}{locdim}
\DeclareMathOperator{\Prim}{Prim}
\DeclareMathOperator{\topdim}{topdim}
\DeclareMathOperator{\dr}{dr}

\title{Covering dimension of Cuntz semigroups}

\author{Hannes Thiel}
\address{H.~Thiel, Institute of Geometry, TU Dresden, 01069 Dresden, Germany.}
\email{hannes.thiel@posteo.de}
\urladdr{www.hannesthiel.org}

\author{Eduard Vilalta}
\address{E.~Vilalta, Departament de Matem\`{a}tiques,
Universitat Aut\`{o}noma de Barcelona,
08193 Bellaterra, Barcelona, Spain}
\email{evilalta@mat.uab.cat}
\urladdr{www.eduardvilalta.com}

\thanks{The first named author was partially supported by the Deutsche Forschungsgemeinschaft (DFG, German Research Foundation) under Germany's Excellence Strategy EXC 2044-390685587 (Mathematics M\"{u}nster: Dynamics-Geometry-Structure) and by the ERC Consolidator Grant No. 681207.
The second named author was partially supported by MINECO (grant No.\ PRE2018-083419 and No.\ MTM2017-83487-P), and by the Comissionat per Universitats i Recerca de la Generalitat de Catalunya (grant No.\ 2017SGR01725).
}
\subjclass[2010]%
{Primary
46L05, % General theory of C*-algebras
46L85; % Noncommutative topology
Secondary
54F45, % Dimension theory
55M10. % Dimension theory
}
\date{\today}

\begin{document}

%==========================================================================================
\begin{abstract}
We introduce a notion of covering dimension for Cuntz semigroups of \ca{s}.
This dimension is always bounded by the nuclear dimension of the \ca, and for subhomogeneous \ca{s} both dimensions agree.

Cuntz semigroups of $\mathcal{Z}$-stable \ca{s} have dimension at most one.
Further, the Cuntz semigroup of a simple, $\mathcal{Z}$-stable \ca{} is zero-dimen\-sional if and only if the \ca{} has real rank zero or is stably projectionless.
\end{abstract}

\maketitle

%==========================================================================================
%==========================================================================================
\section{Introduction}

%==========================================================================================
The Cuntz semigroup of a \ca{} is a powerful invariant in the structure and classification theory of \ca{s}.
It was introduced by Cuntz \cite{Cun78DimFct} in his pioneering work on the structure of simple \ca{s}, and it was a key ingredient for proving the existence of traces on stably finite, exact \ca{s}.

Later, the Cuntz semigroup was used by Toms \cite{Tom08ClassificationNuclear} to distinguish his groundbreaking examples of simple, nuclear \ca{s} that have the same Elliott invariant ($K$-theoretic and tracial data). 
This led to an important revision of Elliott's program to classify simple, nuclear \ca{s}:
To obtain classification by the Elliott invariant, an additional regularity condition is necessary.

The situation was further clarified by the Toms-Winter conjecture (see \cite{Win18ICM} for a recent discussion), which predicts that three regularity conditions are equivalent for simple, nuclear \ca{s}, and that these conditions lead to classification by the Elliott invariant.
The regularity conditions are:
\begin{enumerate}
\item
finite nuclear dimension, where the nuclear dimension \cite{WinZac10NuclDim} is a generalization of covering dimension to nuclear \ca{s};
\item
$\mathcal{Z}$-stability, that is, tensorial absorption of the Jiang-Su algebra $\mathcal{Z}$;
\item
strict comparison of positive elements, a regularity property of the Cuntz semigroup.
\end{enumerate}

It is known that~(1) and~(2) are equivalent \cite{Win12NuclDimZstable, CasEviTikWhiWin21NucDimSimple}, that~(2) implies~(3) \cite{Ror04StableRealRankZ}, and that~(3) implies~(2) in many particular cases (see \cite[Section~9]{Thi20RksOps} for a discussion).
Moreover, by a remarkable breakthrough building on work of numerous people, it is known that unital, separable, simple, nuclear \ca{s} with finite nuclear dimension and satisfying the Universal Coefficient Theorem (UCT) are classified by their Elliott invariant;
see \cite{EllGonLinNiu15arX:classFinDR2} and \cite[Corollary~D]{TikWhiWin17QDNuclear}.

This shows that generalizations of covering dimension for \ca{s} and regularity properties of Cuntz semigroup are closely related and highly relevant for the structure and classification theory of \ca{s}.
In this paper, we combine these aspects by defining a notion of \emph{covering dimension for Cuntz semigroups}, thus introducing a second-level invariant for \ca{s};
see \autoref{dfn:dim}.
More generally, we define covering dimension for abstract Cuntz semigroups, usually called \CuSgp{s}, as introduced in \cite{CowEllIva08CuInv} and extensively studied in \cite{AntPerThi18TensorProdCu, AntPerThi20AbsBivariantCu, AntPerThi20CuntzUltraproducts, AntPerRobThi18arX:CuntzSR1, AntPerRobThi21Edwards}.

Previously, the notion of $n$-comparison for \CuSgp{s} had been considered as `an algebraic interpretation of dimension';
see \cite{Rob11NuclDimComp} and \cite[Remarks~3.2(iii)]{Win12NuclDimZstable}.
However, we believe that the results in this paper show that our notion of dimension is more suited to capture dimensional properties of a \ca{} and its Cuntz semigroup.
For example, for every compact, metrizable space $X$, the \CuSgp{} $\Lsc(X,\NNbar)$ of lower-semicontinuous functions $X\to\NNbar=\{0,1,2,\ldots,\infty\}$ has dimension agreeing with the covering dimension of~$X$;
see \autoref{exa:Lsc}.
More interestingly, we show that a similar result holds for Cuntz semigroups of commutative \ca{s}:

%==========================================================================================
\begin{prpIntro}[{\ref{prp:CommutativeUnital}}]
\label{prpA}
Let $X$ be a compact, Hausdorff space.
Then
\[
\dim(\Cu(C(X))) = \dim(X).
\]
\end{prpIntro}

%==========================================================================================
We prove the expected permanence properties:
The covering dimension does not increase when passing to ideals or quotients of a \CuSgp{} (\autoref{prp:IdealQuot});
the covering dimension of a direct sum of \CuSgp{s} is the maximum of the covering dimensions of the summands (\autoref{prp:IdealQuot});
and if $S=\varinjlim_{\lambda} S_{\lambda}$ is an inductive limit of \CuSgp{s}, then $\dim(\varinjlim_{\lambda} S_{\lambda})\leq \liminf_{\lambda} \dim(S_{\lambda})$ (\autoref{prp:limit}).

In \cite{ThiVil21arX:DimCu2}, we show that the dimension of a \CuSgp{} is determined by the dimensions of its countably based sub-\CuSgp{s}.
This allows us to reduce most questions about dimensions of \CuSgp{s} to the countably based case.
%We also show in \cite{ThiVil21arX:DimCu2} 
It also follows that the dimension of the Cuntz semigroup of a \ca{} is determined by the dimensions of the Cuntz semigroups of its separable sub-\ca{s}.

%==========================================================================================
In \autoref{sec:nuclDim}, we study the connection between the dimension of the Cuntz semigroup of a \ca{} and the nuclear dimension \cite{WinZac10NuclDim} of the \ca{}.

%==========================================================================================
\begin{thmIntro}[{\ref{prp:Nuclear_bound}, \ref{prp:SH}}]
Every \ca{} $A$ satisfies $\dim(\Cu(A))\leq\dimnuc(A)$.
If~$A$ is subhomogeneous, then $\dim(\Cu(A))=\dimnuc(A)$.
\end{thmIntro}

%==========================================================================================
We note that $\dim(\Cu(A))$ can be strictly smaller than $\dimnuc(A)$.
For example, the irrational rotation algebra $A_\theta$ satisfies $\dim(\Cu(A_\theta))=0$ while $\dimnuc(A_\theta)=1$;
see \autoref{exa:DimSmallerDimNuc}.

The dimension of the Cuntz semigroup of a \ca{} $A$ can also be computed in many situations of interest beyond the subhomogeneous case:

\begin{enumerate}
\item
If $A$ has real rank zero, then $\dim(\Cu(A))=0$;
see \autoref{prp:CharDim0UnitalSR1}.
\item
If $A$ is unital and of stable rank one, then $\dim(\Cu(A))=0$ if and only if $A$ has real rank zero;
see \autoref{prp:CharDim0UnitalSR1}.
\item
If $A$ is $\mathcal{Z}$-stable, then $\dim(\Cu(A))\leq 1$;
if $A$ is $\mathcal{W}$-stable, that is, $A$ tensorially absorbs the Jacelon-Razac algebra $\mathcal{W}$, then $\dim(\Cu(A))=0$;
see \autoref{prp:ZstableCAlg}.
\item
If $A$ is purely infinite (not necessarily simple), then $\dim(\Cu(A))=0$;
see \autoref{prp:PurInfCAlg}.
\end{enumerate}

%==========================================================================================
Our results allow us to compute the dimension of the Cuntz semigroup of many simple \ca{s}.
In particular, by \autoref{prp:simple_Zstable}, if $A$ is a separable, simple, $\mathcal{Z}$-stable \ca{}, then
\[
\dim(\Cu(A))=\begin{cases}
0, & \text{ if $A$ has real rank zero or if $A$ is stably projectionless} \\
1, & \text{ otherwise.}
\end{cases}
\]

This should be compared to the computation of the nuclear dimension of a separable, simple \ca{} $A$ as accomplished in \cite{CasEvi20NucDimSimpleProjless, CasEviTikWhiWin21NucDimSimple}:
\[
\dimnuc(A)=\begin{cases}
0, & \text{ if $A$ is an AF-algebra} \\
1, & \text{ if $A$ is nuclear, $\mathcal{Z}$-stable, but not an AF-algebra} \\
\infty, & \text{ if $A$ is nuclear and not $\mathcal{Z}$-stable, or $A$ is not nuclear.}
\end{cases}
\]

%==========================================================================================
It will be interesting to tackle the following problem:

%==========================================================================================
\begin{pbmIntro}
Compute the dimension of the Cuntz semigroups of simple \ca{s}.
In particular, what dimensions can occur (beyond zero and one)?
\end{pbmIntro}

%==========================================================================================
Being low-dimension with respect to a certain dimension theory is often considered as a regularity property.
For example, \ca{s} of stable rank one or real rank zero enjoy properties and admit structure results that do not hold for higher stable or real ranks.
For any notion of dimension it is therefore of much interest to study the objects of lowest dimension.
In the last two sections, we begin such a study for our covering dimension of Cuntz semigroups.
We focus here on the simple case -- the non-simple case will be considered in forthcoming work \cite{ThiVil21arX:ZeroDimCu}.

Given a separable, simple \ca{} $A$ of stable rank one (a large class, which includes many interesting examples - see the introduction of \cite{Thi20RksOps}) such that $\Cu(A)$ is zero-dimensional, we show that $\Cu(A)$ is either algebraic (that is, the compact elements are sup-dense) or soft (that is, $\Cu(A)$ contains no nonzero compact elements);
see \autoref{prp:DichotomySimpleDim0}.
In this setting, compact elements in $\Cu(A)$ correspond to the Cuntz classes of projections in the stabilization $A\otimes\KK$, and nonzero soft elements in $\Cu(A)$ correspond to the Cuntz classes of positive elements in $A\otimes\KK$ with spectrum $[0,1]$;
see \cite{BroCiu09IsoHilbModSF}.

To describe the structure of $\Cu(A)$ in the soft case, we introduce the class of elements with \emph{thin boundary} (see \autoref{dfn:ThinBoundary}), which turn out to play a similar role to that of the compact elements in the algebraic case.
We show that an element~$x$ has thin boundary if and only if it is \emph{complementable} in the sense that for every $y$ satisfying $x\ll y$ there exists $z$ such that $x+z=y$;
see \autoref{prp:ThinIffComplementable}.
Further, the elements with thin boundary form a cancellative monoid;
see \autoref{prp:ThinSummary}.
By combining \autoref{prp:projectionless} with \autoref{prp:CharSimpleSoftDim0}, we obtain:

%==========================================================================================
\begin{thmIntro}
%Let $S$ be a countably based, simple, soft, weakly cancellative \CuSgp{} satisfying \axiomO{5} and \axiomO{6}.
Let $A$ be a separable, simple, stably projectionless \ca{} of stable rank one.
Then $\Cu(A)$ is zero-dimensional if and only if the elements with thin boundary are sup-dense.
\end{thmIntro}

%More explicitly, we prove that a countably based, simple, soft, weakly cancellative $\Cu$-semigroup has dimension zero if and only if the elements with thin boundary are sup-dense; see \autoref{prp:CharSimpleSoftDim0}. 

We finish \autoref{sec:simple} by briefly studying the relation between zero-dimensionality, almost divisibility and the Riesz interpolation property; 
see \autoref{prp:Dim0ImplInterpol}.

%==========================================================================================
\subsection*{Acknowledgements}

The authors want to thank the referee for his thorough reading of this work and for providing valuable comments and suggestions which helped to greatly improve the paper.

%==========================================================================================
%==========================================================================================
\section{Preliminaries}

Let $a,b$ be two positive elements in a \ca{} $A$. Recall that $a$ is said to be \emph{Cuntz subequivalent} to $b$, in symbols $a\precsim b$, if there exists a sequence $(r_n)_n$ in $A$ such that $a=\lim_n r_n b r_n^*$. One defines the equivalence relation $\sim$ by writing $a\sim b$ if $a\precsim b$ and $b\precsim a$, and denotes the equivalence class of $a\in A_+$ by $[a]$.

The Cuntz semigroup of $A$, denoted by $\Cu (A)$, is defined as the quotient of $(A\otimes \mathcal{K})_{+}$ by the equivalence relation $\sim$.
Endowed with the addition induced by
$[a]+[b]=\left[ \begin{psmallmatrix}a&0\\ 0& b\end{psmallmatrix}\right]$ 
and the order induced by $\precsim$, the Cuntz semigroup $\Cu (A)$ becomes a positively ordered monoid.

%=============================================
\begin{pgr}
Given a pair of elements $x,y$ in a partially ordered set, we say that $x$ is \emph{way-below} $y$, in symbols $x\ll y$, if for any increasing sequence $(y_n)_n$ for which the supremum exists and is  greater than $y$ one can find $n$ such that $x\leq y_n$.

It was shown in \cite{CowEllIva08CuInv} that the Cuntz semigroup of any \ca{} satisfies the following properties:
\begin{enumerate}
 \item[\axiomO{1}] Every increasing sequence has a supremum.
 \item[\axiomO{2}] Every element can be written as the supremum of an $\ll$-increasing sequence.
 \item[\axiomO{3}] Given $x'\ll x$ and $y'\ll y$, we have $x'+y'\ll x+y$.
 \item[\axiomO{4}] Given increasing sequences $(x_n)_n$ and $(y_n)_n$, we have $\sup_n x_n +\sup_n y_n= \sup_n (x_n +y_n)$.
\end{enumerate}

In a more abstract setting, any positively ordered monoid satisfying \axiomO{1}-\axiomO{4} is called a \emph{$\Cu$-semigroup}.

A map between two $\Cu$-semigroups is called a \emph{generalized $\Cu$-morphism} if it is a positively ordered monoid homomorphism that preserves suprema of increasing sequences. We say that a generalized $\Cu$-morphism is a \emph{$\Cu$-morphism} if it also preserves the way-below relation. Every *-homomorphism $A\to B$ between \ca{s} naturally induces a $\Cu$-morphism $\Cu (A)\to \Cu (B)$; see \cite[Theorem~1]{CowEllIva08CuInv}. 

We denote by $\Cu$ the category whose objects are \CuSgp{s} and whose morphisms are \CuMor{s}. 

The reader is referred to \cite{CowEllIva08CuInv} and \cite{AntPerThi18TensorProdCu} for a further  detailed exposition.
\end{pgr}

%=============================================
\begin{pgr}
In addition to \axiomO{1}-\axiomO{4}, it was proved in \cite[Proposition~4.6]{AntPerThi18TensorProdCu} and \cite{Rob13Cone} that the Cuntz semigroup of a \ca{} always satisfies the following additional properties:
\begin{enumerate}
\item[\axiomO{5}] 
Given $x+y\leq z$, $x'\ll x$ and $y'\ll y$, there exists $c$ such that $x'+c\leq z\leq x+c$ and $y'\ll c$.
\item[\axiomO{6}]
Given $x'\ll x\leq y+z$ there exist $v\leq x,y$ and $w\leq x,z$ such that $x'\leq v+w$.
\end{enumerate}

Axiom \axiomO{5} is often used with $y=0$. In this case, it states that, given $x'\ll x\leq z$, there exists $c$ such that $x'+c\leq z\leq x+c$.

Recall that a $\Cu$-semigroup is said to be \emph{weakly cancellative} if $x\ll y$ whenever $x+z\ll y+z$ for some element $z$.
Stable rank one \ca{s} have weakly cancellative Cuntz semigroups by \cite[Theorem 4.3]{RorWin10ZRevisited}.
\end{pgr}
%=============================================
\begin{pgr}
A subset $D\subseteq S$ in a \CuSgp{} $S$ is said to be \emph{sup-dense} if whenever $x',x\in S$ satisfy $x'\ll x$, there exists $y\in D$ with $x'\leq y\ll x$.
Equivalently, every element in $S$ is the supremum of an increasing sequence of elements in $D$.

We say that a $\Cu$-semigroup is  \emph{countably based} if it contains a countable sup-dense subset. Cuntz semigroups of separable C*-algebras are countably based (see, for example, \cite{AntPerSan11PullbacksCu}).
\end{pgr}

%==========================================================================================
%==========================================================================================
\section{Dimension of Cuntz semigroups}

%=============================================
In this section we introduce a notion of covering dimension for \CuSgp{s} and study some of its main permanence properties while providing a variety of examples; see  \autoref{prp:IdealQuot}, \autoref{prp:limit} and \autoref{prp:DimRetract}.

In \autoref{prp:DimSoftPart} we investigate the relation between the dimension of a simple \CuSgp{} and its soft part, while in \autoref{prp:RZ_mult} we study how the dimension behaves in the presence of certain  $R$-multiplications. This result is then applied to the Cuntz semigroups of purely infinite, $\mathcal{W}$-stable and $\mathcal{Z}$-stable \ca{s}; see \autoref{prp:PurInfCAlg} and \autoref{prp:ZstableCAlg}.

%=============================================
\begin{dfn}
\label{dfn:dim}
Let $S$ be a \CuSgp.
Given $n\in\NN$, we write $\dim (S)\leq n$ if, whenever $x'\ll x\ll y_1+\ldots +y_r$ in $S$, then there exist $z_{j,k}\in S$ for $j=1,\ldots ,r$ and $k=0,\ldots ,n$ such that: 
\begin{enumerate}[(i)]
\item 
$z_{j,k}\ll y_j$ for each $j$ and $k$;
\item 
$x'\ll \sum_{j,k} z_{j,k}$;
\item 
$\sum_{j=1}^r z_{j,k}\ll x$ for each $k=0,\ldots,n$.
\end{enumerate}

We set $\dim(S)=\infty$ if there exists no $n\in\NN$ with $\dim (S)\leq n$.
Otherwise, we let $\dim(S)$ be the smallest $n\in\NN$ such that $\dim(S)\leq n$.
We call $\dim(S)$ the \emph{(covering) dimension} of $S$.
\end{dfn}

%=============================================
\begin{rmk}
\label{rmk:dim}
Recall that the \emph{(covering) dimension} $\dim(X)$ of a topological space $X$ is defined as the smallest $n\in\NN$ such that every finite open cover of $X$ admits a finite open refinement $\mathcal{V}$ such that at most $n+1$ distinct elements in $\mathcal{V}$ have nonempty intersection;
see for example \cite[Definition~3.1.1, p.111]{Pea75DimThy}.

By \cite[Proposition~1.5]{KirWin04CovDimQD}, a normal space $X$ satisfies $\dim(X)\leq n$ if and only if every finite open cover of $X$ admits a finite open refinement $\mathcal{V}$ that is \emph{$(n+1)$-colorable}, that is, there is a decomposition $\mathcal{V}=\mathcal{V}_0\sqcup\ldots\sqcup\mathcal{V}_n$ such that the sets in $\mathcal{V}_j$ are pairwise disjoint for $j=0,\ldots,n$.
(The sets in $\mathcal{V}_j$ have color $j$, and sets of the same color are disjoint.)

\autoref{dfn:dim} is modeled after the above characterization of covering dimension in terms of colorable refinements.
We interpret the expression `$x\ll y_1+\ldots +y_r$' as saying that $x$ is `covered' by $\{y_1,\ldots,y_r\}$.
Then, condition~(i) from \autoref{dfn:dim} means that $\{z_{j,k}\}$ is a `refinement' of $\{y_1,\ldots,y_r\}$;
condition~(ii) means that $\{z_{j,k}\}$ is a cover of $x'$ (which is an approximation of $x$);
and condition~(iii) means that $\{z_{j,k}\}$ is $(n+1$)-colorable; 
see also \autoref{exa:Lsc} below.
\end{rmk}

%=============================================
In \autoref{dfn:dim}, some of the $\ll$-relations may be changed for $\leq$. 

%=============================================
\begin{lma}
\label{prp:DimWeak}
Let $S$ be a \CuSgp{} and $n\in\NN$.
Then we have $\dim(S)\leq n$ if and only if, whenever $x' \ll x \ll y_1+\ldots+y_r$ in $S$, there exist $z_{j,k}\in S$ for $j=1,\ldots ,r$ and $k=0,\ldots ,n$ such that: 
\begin{enumerate}[(1)]
\item 
$z_{j,k} \leq y_j$ for each $j$ and $k$;
\item 
$x' \leq \sum_{j,k} z_{j,k}$;
\item 
$\sum_{j=1}^r z_{j,k}\leq x$ for each $k=0,\ldots,n$.
\end{enumerate}
\end{lma}
\begin{proof}
The forward implication is clear.
To show the converse, let $x'\ll x\ll y_1+\ldots +y_r$ in $S$.
Choose $s',s,y'_{1},\ldots ,y'_{r}\in S$ such that
\[
x'\ll s'\ll s\ll x\ll y'_{1}+\ldots +y'_{r}, \quad 
y_1'\ll y_1, \quad\ldots, \andSep
y_r'\ll y_r.
\]

Applying the assumption, we obtain elements $z_{j,k}$ for $j=1,\ldots ,r$ and $k=0,\ldots ,n$ satisfying properties (1)-(3) for $s'\ll s\ll y'_{1}+\ldots +y'_{r}$.
Then the same elements satisfy (i)-(iii) in \autoref{dfn:dim} for $x'\ll x\ll y_{1}+\ldots+y_{r}$, thus verifying $\dim(S)\leq n$.
\end{proof}

%=============================================
\begin{exa}
\label{exa:Lsc}
Let $X$ be a compact, metrizable space.
We use $\Lsc(X,\NNbar)$ to denote the set of functions $f\colon X\to\NNbar$ that are lower-semicontinuous, that is, for each $n\in\NN$ the set $f^{-1}(\{n,n+1,\ldots,\infty\})\subseteq X$ is open.
We equip $\Lsc(X,\NNbar)$ with pointwise addition and order.
Then $\Lsc(X,\NNbar)$ is a \CuSgp{};
see, for example, \cite[Corollary~4.22]{Vil21arX:CommCuAI}.
(If $X$ is finite-dimensional, this also follows from 
\cite[Theorem~5.15]{AntPerSan11PullbacksCu}.)
We will show that
\[
\dim(\Lsc(X,\NNbar)) \geq \dim(X).
\]
(The reverse inequality also holds and can be verified through a direct yet elaborate argument.
We defer its proof to \autoref{cor:Lsc}, where we will deduce it easily from the computation of $\dim(C(X))$.) 

Set $n:=\dim(\Lsc(X,\NNbar))$, which we may assume to be finite.
To verify that $\dim(X)\leq n$, let $\mathcal{U}=\{U_{1},\ldots,U_{r}\}$ be a finite open cover of $X$.
We need to find a $(n+1)$-colourable, finite, open refinement of $\mathcal{U}$.

We use $\chi_U$ to denote the characteristic function of a subset $U\subseteq X$.
Given open subsets $U,V\subseteq X$, we have $\chi_U\ll \chi_V$ if and only if $\overline{U}\subseteq V$, that is, $U$ is compactly contained in $V$.
(See the proof of \cite[Corollary~4.22]{Vil21arX:CommCuAI}.)
Then
\[
\chi_X \ll \chi_X \ll \chi_{U_{1}}+\ldots+\chi_{U_{r}}.
\]

Applying that $\dim(\Lsc(X,\NNbar))\leq n$, we obtain elements $z_{j,k}\in\Lsc(X,\NNbar)$ for $j=1,\ldots,r$ and $k=0,\ldots,n$ such that
\begin{enumerate}
\item[(i)]
$z_{j,k}\ll \chi_{U_j}$ for every $j,k$; 
\item[(ii)]
$\chi_X \ll \sum_{j,k} z_{j,k}$; 
\item[(iii)]
$\sum_{j} z_{j,k}\ll \chi_X$ for every $k$.
\end{enumerate}

For each $j$ and $k$, condition~(i) implies that $z_{j,k}=\chi_{V_{j,k}}$ for some open subset $V_{j,k}\subseteq U_j$.
Condition~(ii) implies that $X$ is covered by the sets $V_{j,k}$.
Thus, the family $\mathcal{V}:=\{V_{j,k}\}$ is a finite, open refinement of $\mathcal{U}$.
For each $k$, condition (iii) implies that the sets $V_{1,k},\ldots,V_{r,k}$ are pairwise disjoint.
Thus, $\mathcal{V}$ is $(n+1)$-colourable, as desired.
\end{exa}

%=============================================
Recall that an \emph{ideal} $I$ of a $\Cu$-semigroup $S$ is a downward-hereditary submonoid closed under suprema of increasing sequences; see \cite[Section 5]{AntPerThi18TensorProdCu}.

Given $x,y\in S$, we write $x\leq_I y$ if there exists $z\in I$ such that $x\leq y+z$. We set $x\sim_I y$ if $x\leq_I y$ and $y\leq_I x$. The quotient $S/\sim_I$ endowed with the induced sum and order $\leq_I$ is denoted by $S/I$.

 As shown in \cite[Lemma~5.1.2]{AntPerThi18TensorProdCu}, $S/I$ is a $\Cu$-semigroup and the quotient map $S\to S/I$ is a $\Cu$-morphism.

%=============================================
\begin{prp}
\label{prp:IdealQuot}
Let $S$ be a \CuSgp, and let $I\subseteq S$ be an ideal.
Then:
\[
\dim(I)\leq\dim(S), \andSep
\dim(S/I)\leq\dim (S).
\]
\end{prp}
\begin{proof}
Set $n:=\dim(S)$, which we may assume to be finite, since otherwise there is nothing to prove. 
It is straightforward to show that $\dim(I)\leq n$ using that $I$ is downward-hereditary.
Given $x\in S$, we use $[x]$ to denote its equivalence class in~$S/I$.

To verify $\dim(S/I)\leq n$, let $[u]\ll [x]\ll [y_{1}]+\ldots +[y_{r}]$ in $S/I$.
Then there exists $y_{r+1}\in I$ such that $x\leq y_{1}+\ldots+y_{r}+y_{r+1}$ in $S$. 
Using that the quotient map $S\rightarrow S/I$ preserves suprema of increasing sequences, we can choose $x'',x'\in S$ such that
\[
x''\ll x'\ll x, \andSep
[u]\leq [x''].
\]
Applying the definition of $\dim (S)\leq n$ to $
x''\ll x'\ll y_{1}+\ldots+y_{r}+y_{r+1}$,
we obtain elements $z_{j,k}\in S$ for $j=1,\ldots,r+1$ and $k=0,\ldots,n$ such that $z_{j,k}\ll y_{j}$ for every $j,k$, such that $x''\ll \sum_{j,k} z_{j,k}$, and such that $\sum_{j} z_{j,k}\ll x'$ for every $k$.

Since $y_{r+1}\in I$, we have $z_{r+1,k}\in I$ and thus $[z_{r+1,k}]=0$ in $S/I$ for $k=0,\ldots,n$.
Using also that the quotient map $S\rightarrow S/I$ is $\ll$-preserving, we see that the elements $[z_{j,k}]$ for $j=1,\ldots,r$ and $k=0,\ldots,n$ have the desired properties. 
% $[z_{j,k}]\ll[y_{j}]$ for $j=1,\ldots,r$ and every $k$, that $[u]\ll \sum_k\sum_{j=1}^r [z_{j,k}]$, and that $\sum_{j=1}^r [z_{j,k}]\ll [x]$ for every $k$.
%Since $[a']\ll [d']$ and $[d]\ll a$, it is now easy to check that the elements $[c_{j,k}]$ satisfy properties (i)-(iii) in Definition \ref{dfn:dim} for $[a']\ll [a]\ll [b_{1}]+\cdots +[b_{r}]$. Consequently, $\dim (S/I)\leq\dim (S)$.
\end{proof}

%==========================================================================================
\begin{pbm}
Let $S$ be a \CuSgp, and let $I\subseteq S$ be an ideal.
Can we bound $\dim(S)$ in terms of $\dim(I)$ and $\dim(S/I)$?
In particular, do we always have $\dim(S)\leq\dim(I)+\dim(S/I)+1$?
\end{pbm}

%==========================================================================================
Given \CuSgp{s} $S$ and $T$, we use $S\oplus T$ to denote the Cartesian product $S\times T$ equipped with elementwise addition and order.
It is straightforward to verify that $S\oplus T$ is a \CuSgp{} and that $S\oplus T$ is both the product and coproduct of $S$ and $T$ in the category \CatCu;
see also \cite[Proposition~3.10]{AntPerThi20CuntzUltraproducts}.
We omit the straightforward proof of the next result.

%==========================================================================================
\begin{prp}
\label{prp:sum}
Let $S$ and $T$ be \CuSgp{s}. 
Then
\[
\dim(S\oplus T) = \max\{\dim(S),\dim(T)\}.
\]
\end{prp}

%==========================================================================================
By \cite[Corollary~3.1.11]{AntPerThi18TensorProdCu}, the category $\Cu$ admits inductive limits.
(The sequential case was previously shown in \cite[Theorem~2]{CowEllIva08CuInv}.)
The next result provides a useful characterization of inductive limits in~$\Cu$.

%==========================================================================================
\begin{lma}
\label{prp:CharLimit}
Let $((S_\lambda)_{\lambda\in\Lambda},(\varphi_{\mu,\lambda})_{\lambda\leq\mu \text{ in } \Lambda})$ be an inductive system in~\CatCu, that is, $\Lambda$ is a directed set, each $S_\lambda$ is a \CuSgp{}, and for $\lambda\leq\mu$ in $\Lambda$ we have a connecting \CuMor{} $\varphi_{\mu,\lambda}\colon S_\lambda\to S_\mu$ such that $\varphi_{\lambda,\lambda}=\id_{S_\lambda}$ for every $\lambda\in\Lambda$ and $\varphi_{\nu,\mu}\circ\varphi_{\mu,\lambda}=\varphi_{\nu,\lambda}$ for all $\lambda\leq\mu\leq\nu$ in $\Lambda$.

Then a \CuSgp{} $S$ together with \CuMor{s} $\varphi_{\lambda}\colon S_\lambda\to S$ for $\lambda\in\Lambda$ is the inductive limit in $\Cu$ of the system $((S_\lambda)_{\lambda\in\Lambda},(\varphi_{\mu,\lambda})_{\lambda\leq\mu \text{ in } \Lambda})$ if and only if the following conditions are satisfied:
\begin{enumerate}
\item[(L0)]
we have $\varphi_{\mu}\circ\varphi_{\mu,\lambda}=\varphi_{\lambda}$ for all $\lambda\leq\mu$ in $\Lambda$;
\item[(L1)]
if $x_\lambda',x_\lambda\in S_\lambda$ and $x_\mu\in S_\mu$ satisfy $x_\lambda'\ll x_\lambda$ and $\varphi_{\lambda}(x_\lambda)\leq\varphi_{\mu}(x_\mu)$, then there exists $\nu\geq\lambda,\mu$ such that $\varphi_{\nu,\lambda}(x_\lambda')\ll\varphi_{\nu,\mu}(x_\mu)$;
\item[(L2)]
for all $x',x\in S$ satisfying $x'\ll x$ there exists $x_\lambda\in S_\lambda$ such that  $x'\ll\varphi_\lambda(x_\lambda)\ll x$.
\end{enumerate}
\end{lma}
\begin{proof}
It is shown in \cite[Theorem~2.9]{AntPerThi20CuntzUltraproducts} that \CatCu{} is a full, reflective subcategory of a more algebraic category $\CatW$ defined in \cite[Definition~2.5]{AntPerThi20CuntzUltraproducts}.
The inductive limit in \CatCu{} can therefore be constructed by applying the reflection functor $\CatW\to\CatCu$ to the inductive limit in $\CatW$.

A \emph{$\CatW$-semigroup} is a commutative monoid $S$ together with a transitive, binary relation $\prec$ such that $0\prec x$ for every $x\in S$, such that for every $x\in S$ there is a $\prec$-increasing, $\prec$-cofinal sequence in $x^\prec$, and such that for every $x,y\in S$ the set $x^\prec+y^\prec$ is contained and $\prec$-cofinal in $(x+y)^\prec$;
here, we use the notation $z^\prec := \{ z' : z'\prec z\}$ for $z\in S$, and a subset $C$ of $z^\prec$ is said to be $\prec$-cofinal if for every $z'\in z^\prec$ there exists $c\in C$ with $z'\prec c$.

A \emph{$\CatW$-morphism} between $\CatW$-semigroups $S$ and $T$ is a $\prec$-preserving monoid morphism $\varphi\colon S\to T$ such that for every $x\in S$ the set $\varphi(x^\prec)\subseteq \varphi(x)^\prec$ is $\prec$-cofinal.
Then $\CatW$ is defined as the category of $\CatW$-semigroups and $\CatW$-morphisms.
The inclusion $\CatCu\to\CatW$ is given by mapping a \CuSgp{} $S$ to the underlying monoid of~$S$ together with $\ll$.

To construct the inductive limit of the system $((S_\lambda)_{\lambda\in\Lambda},(\varphi_{\mu,\lambda})_{\lambda\leq\mu \text{ in } \Lambda})$ in $\CatW$, consider the equivalence relation $\sim$ on the disjoint union $\bigsqcup_\lambda S_\lambda$ given by $x_\lambda\sim x_\mu$ (for $x_\lambda\in S_\lambda$ and $x_\mu\in S_\mu$) if there exists $\nu\geq\lambda,\mu$ such that $\varphi_{\nu,\lambda}(x_\lambda)=\varphi_{\nu,\mu}(x_\mu)$.
The set of equivalence classes is the set-theoretic inductive limit, which we denote by $S_{\mathrm{alg}}$.
We write $[x_\lambda]$ for the equivalence class of $x_\lambda\in S_\lambda$.

We define an addition $+$ and a binary relation $\prec$ on $S_{\mathrm{alg}}$ as follows:
Given $x_\lambda\in S_\lambda$ and $x_\mu\in S_\mu$, set
\[
[x_\lambda]+[x_\mu]:=[\varphi_{\nu,\lambda}(x_\lambda)+\varphi_{\nu,\mu}(x_\mu)]
\]
for any $\nu\geq\lambda,\mu$.
Further, set $[x_\lambda]\prec[x_\mu]$ if there exists $\nu\geq\lambda,\mu$ such that $\varphi_{\nu,\lambda}(x_\lambda)\ll\varphi_{\nu,\mu}(x_\mu)$ in $S_\nu$.
This gives $S_{\mathrm{alg}}$ the structure of a $\CatW$-semigroup, which together with the natural maps $S_\lambda\to S_{\mathrm{alg}}, x_\lambda\mapsto[x_\lambda]$, is the inductive limit in $\CatW$.

The reflection of $S_{\mathrm{alg}}$ in $\CatCu$ is a \CuSgp{} $S$ together with a (universal) $\CatW$-morphism $\alpha\colon S_{\mathrm{alg}}\to S$.
Using \cite[Theorem~3.1.8]{AntPerThi18TensorProdCu}, $S$ and $\alpha$ are characterized by the following conditions:
\begin{enumerate}
\item[(R1)]
$\alpha$ is an order-embedding in the sense that $[x_\lambda]^\prec\subseteq[x_\mu]^\prec$ if (and only if) $\alpha([x_\lambda])\leq\alpha([x_\mu])$, for any $x_\lambda\in S_\lambda$ and $x_\mu\in S_\mu$;
\item[(R2)]
$\alpha$ has dense image in the sense that for all $x',x\in S$ satisfying $x'\ll x$ there exists $x_\lambda\in S_\lambda$ such that $x'\ll\alpha([x_\lambda])\ll x$.
\end{enumerate} 

Now the result follows using that the inductive limit of $((S_\lambda)_{\lambda\in\Lambda},(\varphi_{\mu,\lambda})_{\lambda\leq\mu \text{ in } \Lambda})$ in $\CatCu$ is given as the reflection of $S_{\mathrm{alg}}$ in $\CatCu$.
Here, condition~(R1) for the reflection of $S_{\mathrm{alg}}$ in $\CatCu$ corresponds to~(L1), and similarly for (R2) and (L2).
\end{proof}

%==========================================================================================
\begin{prp}
\label{prp:limit}
Let $S=\varinjlim_{\lambda\in\Lambda} S_{\lambda}$ be an inductive limit of \CuSgp{s}. 
Then $\dim(S)\leq \liminf_{\lambda} \dim(S_{\lambda})$.
\end{prp}
\begin{proof}
Let $\varphi_\lambda\colon S_\lambda\to S$ be the \CuMor{s} into the inductive limit.
We use that~$S$ and the $\varphi_\lambda$'s  satisfy (L0)-(L2) from \autoref{prp:CharLimit}.
Set $n:=\liminf_{\lambda} \dim(S_{\lambda})$, which we may assume to be finite. 
To verify $\dim(S)\leq n$, let $x'\ll x\ll y_{1}+\ldots+y_{r}$ in $S$.
Choose $y_1',\ldots,y_r'\in S$ such that
\[
x \ll y_1'+\ldots+y_r', \quad y_1'\ll y_1, \quad\ldots, \andSep y_r'\ll y_r.
\]

Using~(L2), we obtain $a_\lambda\in S_\lambda$ such that $
x' \ll \varphi_\lambda(a_\lambda) \ll x
$. 
Analogously, we obtain $b_{\lambda_k}\in S_{\lambda_k}$ such that $y_k' \ll \varphi_{\lambda_k}(b_{\lambda_k}) \ll y_k$ for $k=1,\ldots,r$.

Using that $\varphi_\lambda$ is a \CuMor{}, we obtain $a_\lambda'',a_\lambda'\in S_\lambda$ such that
\[
x' 
\ll \varphi_\lambda(a_\lambda'') 
\ll \varphi_\lambda(a_\lambda') 
\ll \varphi_\lambda(a_\lambda) 
\ll x, \andSep
a_\lambda''\ll a_\lambda'\ll a_\lambda.
\]

Choose $\mu\in\Lambda$ such that $\mu\geq\lambda,\lambda_1,\ldots,\lambda_r$, and set
\[
a'':=\varphi_{\mu,\lambda}(a_\lambda''),\quad
a':=\varphi_{\mu,\lambda}(a_\lambda'),\quad
a:=\varphi_{\mu,\lambda}(a_\lambda),
\]
and
\[
b_1:=\varphi_{\mu,\lambda_1}(b_{\lambda_1}), \quad\ldots, \andSep
b_r:=\varphi_{\mu,\lambda_r}(b_{\lambda_r}).
\]
Hence,
\[
\varphi_\mu(a)
= \varphi_\mu( \varphi_{\mu,\lambda}(a_\lambda) )
= \varphi_\lambda(a_\lambda) 
\ll x 
\ll y_1'+\ldots+y_r'
\ll \varphi_\mu(b_1+\ldots+b_r).
\]
Applying~(L1), we obtain $\nu\geq\mu$ such that $\varphi_{\nu,\mu}(a') \ll \varphi_{\nu,\mu}(b_1+\ldots+b_r)$.

Using that $\liminf_\lambda \dim(S_{\lambda})\leq n$, we may also assume that $\dim(S_\nu)\leq n$.
Applying $\dim(S_\nu)\leq n$ to
\[
\varphi_{\nu,\mu}(a'')
\ll \varphi_{\nu,\mu}(a')
\ll \varphi_{\nu,\mu}(b_1)+\ldots+\varphi_{\nu,\mu}(b_r),
\]
we obtain elements $z_{j,k}\in S_\nu$ for $j=1,\ldots,r$ and $k=0,\cdots,n$ satisfying properties (i)-(iii) from \autoref{dfn:dim}.
It is now easy to check that the elements $\varphi_{\nu}(z_{j,k})\in S$ have the desired properties to verify $\dim(S)\leq n$.
\end{proof}

%==========================================================================================
\begin{prp}
\label{prp:PermanenceDimCu}
Given a \ca{} $A$ and a closed, two-sided ideal $I\subseteq A$, we have
\[
\dim(\Cu(I)) \leq \dim(\Cu(A)), \andSep
\dim(\Cu(A/I)) \leq \dim(\Cu(A)).
\]

Given \ca{s} $A$ and $B$, we have
\[
\dim(\Cu(A\oplus B)) = \max\{ \dim(\Cu(A)),  \dim(\Cu(B)) \}.
\]

Given an inductive limit of \ca{s} $A=\varinjlim_\lambda A_\lambda$, we have
\[
\dim(\Cu(A)) \leq \liminf_\lambda \dim(\Cu(A_\lambda)).
\]
\end{prp}
\begin{proof}
The first statement follows from \autoref{prp:IdealQuot} using that $\Cu(I)$ is naturally isomorphic to an ideal of $\Cu(A)$, and that $\Cu(A/I)$ is naturally isomorphic to $\Cu(A)/\Cu(I)$;
see \cite[Section~5.1]{AntPerThi18TensorProdCu}.
The second statement follows from \autoref{prp:sum} using that $\Cu(A\oplus B)$ is isomorphic to $\Cu(A)\oplus\Cu(B)$.
Finally, the third statements follows from \autoref{prp:limit} and the fact that the Cuntz semigroup of an inductive limit of \ca{s} is naturally isomorphic to the inductive limit of the \ca{s};
see \cite[Corollary~3.2.9]{AntPerThi18TensorProdCu}.
\end{proof}

%==========================================================================================
\begin{exa}
Recall that $\Cu(\CC)$ is naturally isomorphic to $\NNbar:=\{0,1,2,\ldots,\infty\}$.
We say that a \CuSgp{} $S$ is \emph{simplicial} if $S\cong\NNbar^k=\NNbar\oplus\ldots_k\oplus\NNbar$ for some $k\geq 1$.
If $A$ is a finite-dimensional \ca{}, then $\Cu(A)$ is simplicial.
%A\cong M_{n_1}(\CC)\oplus\ldots\oplus M_{n_k}(\CC)$ for suitable , we have $\Cu(A)\cong\NNbar^k=\NNbar\oplus\ldots_k\oplus\NNbar$.

It is easy to verify that $\dim(\NNbar)=0$.
By \autoref{prp:sum}, we get $\dim(\NNbar^k)=0$ for every $k\geq 1$.
(Using that $\NNbar^k=\Lsc(\{x_1,\ldots,x_k\},\NNbar)$ and that $\{x_1,\ldots,x_k\}$ is zero-dimensional, this also follows from \autoref{cor:Lsc}.)
Thus, if $S$ is an inductive limit of simplicial \CuSgp{s}, then $\dim(S)=0$ by \autoref{prp:limit}.
Further, it follows from \autoref{prp:PermanenceDimCu} that $\dim(\Cu(A))=0$ for every AF-algebra~$A$.
In \autoref{prp:CharDim0UnitalSR1}, we will generalize this to \ca{s} of real rank zero (which include all AF-algebras).

By applying the \CuSgp{} version of the Effros-Handelman-Shen theorem, \cite[Corollary~5.5.13]{AntPerThi18TensorProdCu}, it also follows that every countably-based, weakly cancellative, unperforated, algebraic \CuSgp{} satisfying \axiomO{5} and \axiomO{6} is zero-dimensional.
In \autoref{prp:charAlgDim0}, we will generalize this to weakly cancellative, algebraic \CuSgp{s} satisfying \axiomO{5} and \axiomO{6}.
\end{exa}

%==========================================================================================
\begin{exa}
\label{exa:elementary}
Recall that a \CuSgp{} is said to be \emph{elementary} if it is isomorphic to $\{0\}$, or if it is simple and contains a minimal nonzero element;
see \cite[Paragraph~5.1.16]{AntPerThi18TensorProdCu}.
Typical examples of elementary \CuSgp{s} are $\NNbar$ and $E_k=\{0,1,2,\ldots,k,\infty\}$ for $k\in\NN$, where the sum of two elements in $E_k$ is defined as $\infty$ if their usual sum would exceed $k$;
see \cite[Paragraph~5.1.16]{AntPerThi18TensorProdCu}.
By \cite[Proposition~5.1.19]{AntPerThi18TensorProdCu}, these are the only elementary \CuSgp{s} that satisfy \axiomO{5} and \axiomO{6}.

It is easy to see that every elementary \CuSgp{} satisfying \axiomO{5} and \axiomO{6} is zero-dimensional.
In \autoref{exa:ihomElementary} below, we show that this is no longer the case without \axiomO{5}.
To see that \axiomO{6} is also necessary, consider $S:=\NNbar\cup\{1'\}$, with $1'$ a compact element not comparable with $1$ and such that $1'+1'=2$ and $1+k=1'+k$ for every $k\in\NNbar\setminus\{0\}$.
We claim that $\dim(S)=\infty$.

Assume, for the sake of contradiction, that $\dim (S)\leq n$ for some $n\in\NN$. 
Then, since $1'\ll 1'\ll 2=1+1$, there exist elements $z_{1,k}, z_{2,k}\in S$ for $k=0,\ldots,n$ satisfying conditions (i)-(iii) from \autoref{dfn:dim}.
By condition~(i), we have $z_{j,k}\ll 1$ and therefore $z_{j,k}=0$ or $z_{j,k}=1$ for every $j,k$.
By condition~(ii), we have $1'\ll\sum_{j,k}z_{j,k}$, and so there exist $j'\in\{1,2\}$ and $k'\in\{0,\ldots,n\}$ such that $z_{j',k'}=1$.
However, by condition~(iii), we have $z_{j',k'}\ll 1'$, which is a contradiction because the elements $1$ and $1'$ are not comparable.
\end{exa}

%==========================================================================================
\begin{exa}
\label{exa:ihomElementary}
Let $k,l\in\NN$, and let $E_k$ and $E_l$ be the elementary \CuSgp{s} as in \autoref{exa:elementary}.
Then the abstract bivariant \CuSgp{} $\ihom{E_{k},E_{l}}$, as defined in \cite{AntPerThi20AbsBivariantCu}, has dimension one whenever $l>k$ and dimension zero otherwise.

Indeed, by \cite[Proposition~5.18]{AntPerThi20AbsBivariantCu}, we know that $\ihom{E_{k},E_{l}}=\{0,r,\ldots,l,\infty\}$ with $r=\lceil (l+1)/(k+1)\rceil$.
Thus, if $l\leq k$, then $\ihom{E_{k},E_{l}}=E_l$, which is zero-dimensional by \autoref{exa:elementary}.
Note that $\ihom{E_{k},E_{l}}$ is an elementary \CuSgp{} satisfying \axiomO{6}.
Further, $\ihom{E_{k},E_{l}}$ satisfies \axiomO{5} if and only if $l\leq k$.

Let us now assume that $l>k$, that is $r>1$.
Then, even though $r+1\ll r+1\ll r+r$, one cannot find $z_{1},z_{2}\ll r$ such that $r+1=z_{1}+z_{2}$.
This shows that $\dim(\ihom{E_{k},E_{l}})\neq 0$.
 
To verify $\dim(\ihom{E_{k},E_{l}})\leq 1$, let $x\ll x\ll y_{1}+\ldots+y_{r}$ in $\ihom{E_{k},E_{l}}$.
We may assume that $y_{j}$ is nonzero for every $j$. 
If there exists $i\in\{1,\ldots,r\}$ with $x\leq y_{i}$, then $z_{i,0}:=x$ and $z_{j,k}:=0$ for $j\neq i$ or $k=1$ have the desired properties.

So we may assume that $y_j<x$ for every $j$.
Let $k$ be the least integer such that $x\leq y_{1}+\ldots+y_{k}$.
Define $z_{j,0}:=y_{j}$ for every $j<k$ and $z_{j,0}:=0$ for $j\geq k$. 
Further, define $z_{k,1}:=y_{k}$ and $z_{j,1}:=0$ for $j \neq k$. 
By choice of $k$, we have $\sum_{j}z_{j,0}\ll x$.
We also have $\sum_{j}z_{j,1}=y_k\ll x$.
Finally, $x\ll \sum_{j}z_{j,0} + \sum_{j}z_{j,1}$, as desired.
\end{exa}

%==========================================================================================
\begin{dfn}
\label{dfn:retract}
Let $S$ and $T$ be \CuSgp{s}.
We say that $S$ is a \emph{retract} of $T$ if there exist a \CuMor{} $\iota\colon S\to T$ and a generalized \CuMor{} $\sigma\colon T\to S$ such that $\sigma\circ\iota=\id_S$.
\end{dfn}

%==========================================================================================
Many properties of \CuSgp{s} pass to retracts.
In \autoref{prp:RetractInterpol} we show this for the Riesz interpolation property and for almost divisibility.
The next result shows that the dimension does not increase when passing to a retract.

%==========================================================================================
\begin{prp}
\label{prp:DimRetract}
Let $S$ and $T$ be \CuSgp{s} and assume that $S$ is a retract of $T$.
Then $dim(S)\leq\dim(T)$.
\end{prp}
\begin{proof}
Let $\iota\colon S\to T$ be a \CuMor, and let $\sigma\colon T\to S$ be a generalized \CuMor{} such that $\sigma\circ\iota=\id_S$.
Set $n:=\dim(T)$, which we may assume to be finite, and let $x' \ll x \ll y_1+\ldots+y_r$ in $S$. 
Then
\[
\iota(x') \ll \iota(x) \ll \iota(y_1) + \ldots + \iota(y_r)
\]
in $T$.
Using that $\dim(T)\leq n$, we obtain elements $z_{j,k}$ in $T$ satisfying conditions (i)-(iii) of \autoref{dfn:dim}.
Applying $\sigma$, we see that the elements $\sigma (z_{j,k})$ satisfy conditions (1)-(3) in \autoref{prp:DimWeak}, from which the result follows.
\end{proof}

%==========================================================================================
Given a simple \CuSgp{} $S$, let us now show that its sub-\CuSgp{} of soft elements $S_\soft$, as defined in \autoref{pgr:soft}, is a retract of $S$. 
As we will see in \autoref{prp:DimSoftPart} below, such elements play an important role in the study of the dimension of $S$.

%==========================================================================================
\begin{prp}
\label{prp:retractSimpleSoft}
Let $S$ be a countably based, simple, weakly cancellative \CuSgp{} satisfying \axiomO{5} and \axiomO{6}.
Then $S_\soft$ is a retract of $S$.
\end{prp}
\begin{proof}
By \cite[Proposition~5.3.18]{AntPerThi18TensorProdCu}, $S_\soft$ is a \CuSgp{}.
By \cite[Proposition~2.9]{Thi20RksOps}, for each $x\in S$ there exists a (unique) maximal soft element dominated by $x$ and the map $\sigma\colon S\to S_\soft$ given by
\[
\sigma(x) := \max \big\{ x'\in S_\soft : x' \leq x \big\}, \quad\text{ for $x\in S$,}
\]
is a generalized \CuMor.
Further, the inclusion $\iota\colon S_\soft\to S$ is a \CuMor{} and the composition $\sigma\circ\iota$ is the identity on $S_\soft$, as desired.
\end{proof}

%==========================================================================================
\begin{prp}
\label{prp:DimSoftPart}
Let $S$ be a countably based, simple, weakly cancellative \CuSgp{} satisfying \axiomO{5} and \axiomO{6}.
Then
\[
\dim(S_\soft)\leq\dim(S)\leq\dim(S_\soft)+1.
\]
\end{prp}
\begin{proof}
The first inequality follows from Propositions  \ref{prp:DimRetract} and \ref{prp:retractSimpleSoft}.
To show the second inequality, set $n:=\dim(S_\soft)$, which we may assume to be finite.
If $S$ is elementary, then $\dim(S)=0$ as noted in \autoref{exa:elementary}.
Thus, we may assume that $S$ is nonelementary.
By \cite[Proposition~5.3.16]{AntPerThi18TensorProdCu}, every nonzero element of $S$ is either soft or compact.
To verify $\dim(S)\leq n+1$, let $x'\ll x\ll y_1+\ldots+y_r$ in $S$.
We may assume that $x$ and $y_1$ are nonzero.
If $x$ is soft, then we let $s',s\in S$ be any pair of soft elements satisfying $x'\ll s'\ll s\ll x$.
If $x$ is compact, then we apply \autoref{prp:smallElements} to obtain a nonzero element $w\in S$ satisfying $w\leq x,y_1$.
Then $x\ll\sigma(x)+w$, which allows us to choose soft elements $s'\ll s$ such that $ s\ll\sigma(x)$ and $x\ll s'+w$.
In both cases, we have
\[
s'\ll s \ll \sigma(x) \leq \sigma(y_1)+\ldots+\sigma(y_r)
\]
in $S_\soft$.
Using that $\dim(S_\soft)\leq n$, we obtain (soft) elements $z_{j,k}\in S$ for $j=1,\ldots,r$ and $k=0,\ldots,n$ such that
\begin{enumerate}[(i)]
\item 
$z_{j,k}\ll \sigma(y_j)$ (and thus, $z_{j,k}\ll y_j$) for each $j$ and $k$;
\item 
$s'\ll \sum_{j,k} z_{j,k}$;
\item 
$\sum_{j=1}^r z_{j,k}\ll s\ll \sigma(x)$ (and thus, $\sum_{j=1}^r z_{j,k}\ll x$) for each $k=0,\ldots,n$.
\end{enumerate}
If $x$ is soft, then $x'\ll s' \ll \sum_{j,k} z_{j,k}$, which shows that the elements $z_{j,k}$ have the desired properties.
If $x$ is compact, then set $z_{1,n+1}:=w$ and $z_{j,n+1}:=0$ for $j=2,\ldots,r$.
Then $z_{j,n+1}\ll y_j$ for each $j$.
Further,
\[
x'\ll x\ll s'+w \leq \left( \sum_{k=0}^n\sum_{j=1}^r z_{j,k} \right) + \sum_{j=1}^r z_{j,n+1}
= \sum_{k=0}^{n+1}\sum_{j=1}^r z_{j,k}.
\]
Lastly, $\sum_{j=1}^r z_{j,n+1}=w\ll x$, which shows that the elements $z_{j,k}$ have the desired properties.
\end{proof}

%==========================================================================================
\begin{rmk}
\autoref{prp:DimSoftPart} applies in particular to the Cuntz semigroups of separable, simple \ca{s} of stable rank one (see \cite[Proposition~5.1.1]{Rob13Cone}).
More generally, Engbers showed in \cite{Eng14PhD} that for every separable, simple, stably finite \ca{} $A$, every compact element in $\Cu(A)$ has a predecessor.
The proof of \autoref{prp:DimSoftPart} can be generalized to this situation and we obtain 
\[
\dim(\Cu(A)_\soft)\leq\dim(\Cu(A))\leq\dim(\Cu(A)_\soft)+1.
\]
\end{rmk}

%==========================================================================================
\begin{exa}
\label{exa:Z}
Let $Z=\Cu(\mathcal{Z})$, the Cuntz semigroup of the Jiang-Su algebra~$\mathcal{Z}$.
Then, $\dim(Z)=1$.
Indeed, since $Z_\soft\cong[0,\infty]$, and it is easy to verify that $\dim([0,\infty])=0$, we have $\dim(Z)\leq\dim([0,\infty])+1=1$ by \autoref{prp:DimSoftPart}. 

On the other hand, to see that $\dim (Z)\neq 0$, consider the compact element $1\in Z$ and the soft element $\tfrac{3}{4}\in Z$.
Then $1\ll 1\ll \tfrac{3}{4}+\tfrac{3}{4}$, but there are no elements $z_0,z_1\in Z$ satisfying $1=z_0+z_1$ and $z_0,z_1\ll\tfrac{3}{4}$, which shows that $\dim(Z)\neq 0$.
(It also follows from \autoref{prp:DichotomySimpleDim0} that $Z$ is not zero-dimensional.)

A similar argument shows that $\dim(Z')=1$, where $Z'$ is the \CuSgp{} considered in \cite[Question~9(8)]{AntPerThi18TensorProdCu}, that is, $Z':=Z\cup\{ 1'' \}$ with $1''$ a compact element not comparable with $1$ and such that $1''+1''=2$ and $1+x=1''+x$ for every $x\in Z\setminus\{ 0\}$.
%Note that this example shows, in particular, that there are Cu-semigroups of finite dimension not satisfying (O6).
\end{exa}

%==========================================================================================
The notion of $R$-multiplication on a \CuSgp{} for a $\CatCu$-semiring $R$ was introduced in \cite[Definition~7.1.3]{AntPerThi18TensorProdCu}. 
Informally, an $R$-multiplication is a scalar multiplication on the semigroup with natural compatibility conditions. 
Given a solid $\CatCu$-semiring $R$ (such as $\{ 0,\infty\}$, $[0,\infty]$ or~$Z$), any two $R$-multiplications on a \CuSgp{} are equal, and therefore having an $R$-multiplication is a property;
see \cite[Remark~7.1.9]{AntPerThi18TensorProdCu}.

It was shown in \cite[Theorem~7.2.2]{AntPerThi18TensorProdCu} that a \CuSgp{} has $\{ 0,\infty \}$-multiplication if and only if every element in the semigroup is idempotent.
By \cite[Theorem~7.3.8]{AntPerThi18TensorProdCu}, a \CuSgp{} has $Z$-multiplication if and only if it is almost unperforated and almost divisible.
By \cite[Theorem~7.5.4]{AntPerThi18TensorProdCu}, a \CuSgp{} has $[0,\infty]$-multiplication if and only if it has $Z$-multiplication and every element in $S$ is soft.

%==========================================================================================
\begin{prp}
\label{prp:RZ_mult}
Let $S$ be a \CuSgp{} satisfying \axiomO{5} and \axiomO{6}.
Then:
\begin{enumerate}
\item
If $S$ has $\{ 0,\infty \}$-multiplication, then $\dim(S)=0$.
\item
If $S$ has $[0,\infty]$-multiplication, then $\dim(S)=0$.
\item 
If $S$ has $Z$-multiplication, then $\dim(S)\leq 1$.
\end{enumerate}
\end{prp}
\begin{proof}
(1)
Given elements $x'\ll x\ll y_1+\ldots +y_r$ in a \CuSgp{} with $\{0,\infty \}$-multiplication, apply \axiomO{6} to obtain elements $z_j\leq x,y_j$ such that
\[
 x'\leq z_1+\ldots +z_r.
\]

Using that every element in $S$ is idempotent, one also has
\[
 z_1+\ldots +z_r\leq x+\ldots_r +x=rx=x.
\]

This shows that the elements $z_j$ satisfy the conditions in Lemma \ref{prp:DimWeak}, as required.

(2)
Note that $S$ is isomorphic to its realification $S_R$ by Theorem~7.5.4 and Proposition~7.5.9 in \cite{AntPerThi18TensorProdCu}. 
We can now use the decomposition property of $S_{R}$ proven in \cite[Theorem~4.1.1]{Rob13Cone} to deduce that $S$ is zero-dimensional.

(3)
Assume that $S$ has $Z$-multiplication.
By \cite[Proposition~7.3.13]{AntPerThi18TensorProdCu}, an element $x\in S$ is soft if and only if $x=1'x$ (where $1'$ denotes the soft one in $Z$).
Further, the \CuSgp{} $S_\soft:=1' S$ of soft elements in $S$ is isomorphic to the realification of $S$;
see \cite[Corollary~7.5.10]{AntPerThi18TensorProdCu}.
Since the realification of $S$ has $[0,\infty]$-multiplication, we get $\dim(S_\soft)=0$ by (2).

To verify $\dim(S)\leq 1$, let $
x'\ll x\ll y_{1}+\ldots+y_{r}$ 
in $S$.
Using that $S$ has $Z$-multiplication, one gets
\[
\tfrac{5}{8} x' 
\ll \tfrac{6}{8} x 
\ll \tfrac{7}{8} y_1 + \ldots + \tfrac{7}{8} y_r.
\]
Note that all elements in the previous expression belong to $S_\soft$.
Since $\dim(S_\soft)=0$, we obtain (soft) elements $z_1,\ldots,z_r\in S$ such that $z_j\ll \tfrac{7}{8} y_{j}$ for each $j$, and such that 
\[                                                               
\tfrac{5}{8} x' 
\ll z_{1}+\ldots+z_{r}
\ll \tfrac{6}{8} x.
\]

Define $z_{j,0}:=z_{j}$ and $z_{j,1}:=z_{j}$ for $j=1,\ldots,r$.
We trivially have $z_{j,k}\ll y_{j}$ for each $j$ and $k$.
Further,
\[
x' 
\leq \tfrac{10}{8}x'
\ll 2(z_{1}+\cdots +z_{r})
= \sum_{j,k}z_{j,k},
\]
and
\[
\sum_{j}z_{j,k} \ll \tfrac{6}{8}x \leq x.
\]
for each $k=0,1$, as desired.
\end{proof}

%==========================================================================================
Let $A$ be a \ca{}. 
Then, we know from \cite[Proposition~7.2.8]{AntPerThi18TensorProdCu} that $A$ is purely infinite if and only if $\Cu (A)$ has $\{ 0,\infty \}$-multiplication.

%==========================================================================================
\begin{prp}
\label{prp:PurInfCAlg}
Let $A$ be a purely infinite \ca{}. 
Then $\dim (\Cu (A))=0$.
\end{prp}

%==========================================================================================
Let $\mathcal{W}$ denote the Jacelon-Racak algebra.
Given a \ca{} $A$, it follows from \cite[Proposition~7.6.3]{AntPerThi18TensorProdCu} that $\Cu(A\otimes\mathcal{W})$ has $[0,\infty]$-multiplication, and that $\Cu(A\otimes\mathcal{Z})$ has $Z$-multiplication.

%==========================================================================================
\begin{prp}
\label{prp:ZstableCAlg}
Let $A$ be a \ca{}.
Then
\[
\dim(\Cu(A\otimes\mathcal{W}))=0, \andSep
\dim(\Cu(A\otimes\mathcal{Z}))\leq 1.
\]
In particular, Cuntz semigroups of $\mathcal{W}$-stable \ca{s} are zero-dimensional, and Cuntz semigroups of $\mathcal{Z}$-stable \ca{s} have dimension at most one.
\end{prp}

%==========================================================================================
\begin{exa}
Let $X$ be a compact, metrizable space containing at least two points, and let $S:=\Lsc(X,\NNbar)_{++}\cup\{ 0\}$ be the sub-\CuSgp{} of $\Lsc(X,\NNbar)$ consisting of strictly positive functions and $0$.
Then $\dim(S)=\infty$. % unless $X$ is a point, in which case $\dim (S)=0$ because $S=\NNbar$.

Indeed, assume for the sake of contradiction that $\dim(S)\leq n$ for some $n\in\NN$, and take $r>n$.
Since $X$ contains at least two points, we can choose open subsets $U',U\subset X$ such that
\[
\emptyset\neq U', \quad
\overline{U'}\subseteq U, \andSep
U\neq X.
\]
Let $\chi_{U'}$ and $\chi_U$ denote the corresponding characteristic functions.
Consider the elements $x':=1+(n+1)\chi_{U'}$ and $x:=1+(n+1)\chi_{U}$ in $S$. 
Then, we have $x'\ll x\ll r+1=1+\ldots_{r+1}+1$ in $S$.
 
Using that $\dim(S)\leq n$, we obtain elements $z_{j,k}\in S$ for $j=1,\ldots,r+1$ and $k=0,\ldots,n$ satisfying (i)-(iii) from \autoref{dfn:dim}.
By condition~(i), we have $z_{j,k}\ll 1$ and therefore $z_{j,k}=0$ or $z_{j,k}=1$ for each $j,k$.

Given $k\in\{0,\ldots,n\}$, we have $\sum_{j}z_{j,k}\ll x$ by condition~(iii), and thus all but possibly one of the elements $z_{1,k},\ldots,z_{r,l}$ are zero.
Thus, $\sum_{j}z_{j,k}\leq 1$.
Using this at the last step, and using condition~(ii) at the first step, we get
\[
x'\ll \sum_{j,k}z_{j,k} = \sum_{k=0}^n\big( \sum_{j=1}^r z_{j,k} \big) \leq n+1,
\]
a contradiction.

Note that $S$ does not arise as the Cuntz semigroup of a \ca{} since it does not satisfy \axiomO{5}.
(Take, for instance, $1\ll 1\ll 1+\chi_U$ with $U\neq X$.)
\end{exa}

%==========================================================================================
%==========================================================================================
\section{Commutative and subhomogeneous \texorpdfstring{$C^*$}{C*}-algebras}
\label{sec:nuclDim}

%==========================================================================================
In this section, we first prove that the dimension of the Cuntz semigroup of a \ca{} $A$ is bounded by the nuclear dimension of $A$; 
see \autoref{prp:Nuclear_bound}.
For every compact, Hausdorff space $X$, we show that the dimension of the Cuntz semigroup of $C(X)$ agrees with the dimension of $X$;
see \autoref{prp:CommutativeUnital}.
More generally, on the class of subhomogeneous \ca{s}, the dimension of the Cuntz semigroup agrees with the topological dimension, which in turn is equal to the nuclear dimension;
see \autoref{prp:SH}.

%==========================================================================================
\begin{thm}
\label{prp:Nuclear_bound}
Let $A$ be a \ca.
Then $\dim(\Cu(A))\leq\dimnuc(A)$.
\end{thm}
\begin{proof}
Set $n:=\dimnuc(A)$, which we may assume to be finite.
By \cite[Proposition~2.2]{Rob11NuclDimComp}, there exists an ultrafilter $\mathcal{U}$ on an index set $\Lambda$, and finite-dimensional C*-algebras $F_{\lambda,k}$ for $\lambda\in\Lambda$ and $k=0,\ldots,n$, and completely positive, contractive (cpc.)\ order-zero maps $\psi_k\colon A\to \prod_{\mathcal{U}}F_{\lambda,k}$ and $\varphi_k\colon \prod_{\mathcal{U}}F_{\lambda,k}\to A_{\mathcal{U}}$ such that
\[
\iota = \sum_{k=0}^n \varphi_k\circ\psi_k,
\]
where $\iota\colon A\to A_{\mathcal{U}}$ denotes the natural inclusion map.

A cpc.\ order-zero map $\alpha\colon C\to D$ between \ca{s} induces a generalized \CuMor{} $\bar{\alpha}\colon\Cu(C)\to\Cu(D)$; 
see, for example, \cite[Paragraph~3.2.5]{AntPerThi18TensorProdCu}.

The equality $\iota = \sum_{k=0}^n \varphi_k\circ\psi_k$ implies that
\[
\bar{\varphi}_l(\bar{\psi}_l(x))
\leq \bar{\iota}(x)
\leq \sum_{k=0}^n \bar{\varphi}_k(\bar{\psi}_k(x))
\]
for each $x\in\Cu(A)$ and each $l\in\{0,\ldots,n\}$.

To verify $\dim(\Cu(A))\leq n$, let $x',x,y_1,\ldots,y_r\in\Cu(A)$ satisfy
\[
x' \ll x \ll y_1+\ldots+y_r.
\]

For each $k\in\{0,\ldots,n\}$, set $x_k := \bar{\psi}_k(x) \in \Cu(\prod_{\mathcal{U}}F_{\lambda,k})$.
We have
\[
\bar{\iota}(x')
\ll \bar{\iota}(x) 
\leq \sum_{k=0}^n \bar{\varphi}_k(\bar{\psi}_k(x))
= \sum_{k=0}^n \bar{\varphi}_k(x_k).
\]

Using that $\bar{\varphi}_k$ preserves suprema of increasing sequences, we can choose an element $x_k'\in\Cu(\prod_{\mathcal{U}}F_{\lambda,k})$ such that $x_k'\ll x_k$ and 
\[
\bar{\iota}(x')
\ll \sum_{k=0}^n \bar{\varphi}_k(x_k').
\]

Given $k\in\{0,\ldots,n\}$, we have
\[
x_k' 
\ll x_k 
= \bar{\psi}_k(x) 
\leq \bar{\psi}_k(\sum_{j=1}^r y_j)
= \sum_{j=1}^r \bar{\psi}_k(y_j).
\]
Since $\prod_{\mathcal{U}}F_{\lambda,k}$ has real rank zero, we obtain $z_{1,k},\ldots,z_{r,k}\in\Cu(\prod_{\mathcal{U}}F_{\lambda,k})$ such that $z_{j,k}\leq \bar{\psi}_k(y_j)$ for $j=1,\ldots,r$ and
\[
x_k' \leq \sum_{j=1}^r z_{j,k} \leq x_k.
\]

We now consider the elements $\bar{\varphi}_k(z_{j,k})\in\Cu(A_{\mathcal{U}})$.
For each $j$ and $k$, we have
\[
\bar{\varphi}_k( z_{j,k} )
\leq \bar{\varphi}_k( \bar{\psi}_k(y_j) )
\leq \bar{\iota}(y_j).
\]
Further, we have
\[
\bar{\iota}(x')
\ll \sum_{k=0}^n \bar{\varphi}_k(x_k')
\leq \sum_{k=0}^n \bar{\varphi}_k( \sum_{j=1}^r z_{j,k} )
= \sum_{k=0}^n \sum_{j=1}^r \bar{\varphi}_k( z_{j,k} ).
\]
For each $k\in\{0,\ldots,n\}$, we also have
\[
\sum_{j=1}^r \bar{\varphi}_k( z_{j,k} )
= \bar{\varphi}_k( \sum_{j=1}^r z_{j,k} )
\leq \bar{\varphi}_k( x_k )
= \bar{\varphi}_k( \bar{\psi}_k(x) )
\leq \bar{\iota}(x).
\]

Since the classes of elements in $\bigcup_{N\in\NN}(A_{\mathcal{U}}\otimes M_N )_{+}$ are sup-dense in $\Cu (A_{\mathcal{U}})$, there exist $N\in\NN$ and positive elements $c_{j,k}\in A_{\mathcal{U}}\otimes M_N$ such that $[c_{j,k}]\ll \bar{\varphi}_k( z_{j,k} )$ and $\bar{\iota}(x')
\ll \sum_{j,k}[c_{j,k}]$.

We have $A_{\mathcal{U}} = \prod_\lambda A / c_{\mathcal{U}}$, where
\[
c_{\mathcal{U}} = \{ (a_\lambda)_\lambda \in \prod_\lambda A : \lim_{\lambda\to\mathcal{U}}\|a_\lambda\|=0 \}.
\]
We let $\pi\colon \prod_\lambda A \to A_{\mathcal{U}}$ denote the quotient map.

We have $A_{\mathcal{U}}\otimes M_N \cong (A\otimes M_N)_{\mathcal{U}}$.
We also use $\pi$ to denote its amplification to matrix algebras.
Choose positive elements $c_{j,k,\lambda}\in A\otimes M_N$ such that $\pi((c_{j,k,\lambda})_\lambda) = c_{j,k}$.
Then, for a sufficiently large $\lambda$, the elements $[c_{j,k,\lambda}]\in\Cu(A)$ satisfy the conditions of Lemma \ref{prp:DimWeak} for $x'\ll x\ll y_1+\cdots +y_r$, as desired.
\end{proof}

%==========================================================================================
We will show in \autoref{prp:SH} that for subhomogeneous \ca{s}, the dimension of the Cuntz semigroup is \emph{equal} to the nuclear dimension.
To prove this, we first apply some model theoretic techniques to reduce the problem to separable, subhomogeneous \ca{s};
see \autoref{prp:SubSepSH}.

%==========================================================================================
\begin{pgr}
Recall that the \emph{local dimension} $\locdim(X)$ of a topological space $X$ is defined as the smallest $n\in\NN$ such that every point in $X$ has a closed neighborhood of covering dimension at most $n$;
see \cite[Definition~5.1.1, p.188]{Pea75DimThy}.
If $X$ is a locally compact, Hausdorff space, then
\[
\locdim(X) = \sup \big\{ \dim(K) : K\subseteq X \text{ compact } \big\}.
\] 
If $X$ is $\sigma$-compact, locally compact and Hausdorff, then $\locdim(X)=\dim(X)$, but in general $\locdim(X)$ can be strictly smaller than $\dim(X)$. % (even for locally compact, Hausdorff spaces).
If $X$ is locally compact, Hausdorff but not compact, then it follows from \cite[Proposition~3.5.6]{Pea75DimThy} that $\locdim(X)$ agrees with the the dimension of $\alpha X$, the one-point compactification of~$X$.
\end{pgr}

%==========================================================================================
\begin{pgr}
\label{pgr:SH}
Let $d\in\NN$ with $d\geq 1$.
Recall that a \ca{} $A$ is said to be \emph{$d$-(sub)homo\-ge\-neous} if every irreducible representation of $A$ has dimension (at most) $d$.
Further, $A$ is \emph{(sub)homogeneous} if it is $d$-(sub)homogeneous for some $d$.
If $A$ is $d$-subhomogeneous, then so is every sub-\ca{} of $A$. Let us briefly recall the main structure theorems for (sub)homogeneous \ca{s}.
For details, we refer to \cite[Sections~IV.1.4, IV.1.7]{Bla06OpAlgs}.

Given a locally trivial $M_d(\CC)$-bundle over a locally compact, Hausdorff space $X$, the algebra of sections vanishing at infinity is a $d$-homogeneous \ca{} with primitive ideal space homeomorphic to $X$.
Moreover, every homogeneous \ca{} arises this way.

Let $A$ be a $d$-subhomogeneous \ca.
For each $k\geq 2$, let $I_{\geq k}\subseteq A$ be the set of elements $a\in A$ such that $\pi(a)=0$ for every irreducible representation $\pi$ of $A$ of dimension at most $k-1$.
Set $I_{\geq 1} = A$.
Then
\[
\{0\}=I_{\geq d+1} \subseteq I_{\geq d} \subseteq \ldots \subseteq I_{\geq 2} \subseteq I_{\geq 1} = A
\]
is an increasing chain of (closed, two-sided) ideals of $A$.
For each $k\geq 1$, the canonical $k$-homogeneous ideal-quotient (that is, an ideal of a quotient) of $A$ is
\[
A_k := I_{\geq k}/I_{\geq k+1}.
\]
Note that $A_k=\{0\}$ for $k\geq d+1$.

For each $k\geq 1$, we have a short exact sequence
\[
0 \to A_{k+1} \to A/I_{\geq k+1} \to A/I_{\geq k} \to 0.
\]
In particular, $A/I_{\geq 3}$ is an extension of $A/I_{\geq 2}=A_1$ by $A_2$.
Then $A/I_{\geq 4}$ is an extension of $A/I_{\geq 3}$ by $A_3$, and so on.
Finally, $A$ is an extension of $A/I_{\geq d-1}$ by $A_d$.
Thus, every subhomogeneous \ca{} is obtained as a finite sequence of successive extensions of homogeneous \ca{s}.

In \cite{BroPed09Limits}, Brown and Pedersen introduced the \emph{topological dimension} for certain \ca{s}, including all type~$\mathrm{I}$ \ca{s}.
We only recall the definition for subhomogeneous \ca{s}.
First, if $A$ is homogeneous, then its primitive ideal space $\Prim(A)$ is locally compact and Hausdorff, and then the topological dimension of $A$ is defined as $\topdim(A):=\locdim(\Prim(A))$.

If $A$ is subhomogeneous, then the topological dimension of $A$ is defined as the maximum of the topological dimensions of the canonical homogeneous ideal-quotients:
% $A_k$ for $k=1,\ldots,d$ defined in \autoref{pgr:SH}:
\[
\topdim(A) 
:= \max_{k=1,\ldots,d} \topdim(A_k)
= \max_{k=1,\ldots,d} \locdim(\Prim(A_k)).
\]
\end{pgr}

%==========================================================================================
Given a \ca{} $A$, we use $\SubSep(A)$ to denote the collection of separable sub-\ca{s} of $A$.
A family $\mathcal{S}\subseteq\SubSep(A)$ is said to be \emph{$\sigma$-complete} if for every countable, directed subfamily $\mathcal{T}\subseteq\mathcal{S}$ we have $\overline{\bigcup\{B:B\in\mathcal{T}\}}\in\mathcal{S}$.
Further, a family $\mathcal{S}\subseteq\SubSep(A)$ is said to be \emph{cofinal} if for every $B_0\in\SubSep(A)$ there exists $B\in\mathcal{S}$ with $B_0\subseteq B$.

%==========================================================================================
\begin{prp}
\label{prp:SubSepSH}
Let $n\in\NN$.
Then for every subhomogeneous \ca{} $A$ satisfying $\topdim(A)\leq n$, the set
\[
\big\{ B\in\SubSep(A) : \topdim(B)\leq n \big\}
\]
is $\sigma$-complete and cofinal.
\end{prp}
\begin{proof}
We will use the following facts.
The first is a consequences of \cite[Proposition~3.5]{Thi20arX:grSubhom}, the second follows from \cite[Proposition~2.2]{BroPed09Limits}.

\textbf{Fact~1:}
\emph{Given a homogeneous \ca{} $B$ with $\locdim(B)\leq n$, the collection
\[
\big\{ C\in\SubSep(B) : \topdim(C)\leq n \big\}
\]
is $\sigma$-complete and cofinal.}

\textbf{Fact~2:}
\emph{If $B$ is subhomogeneous and $I\subseteq B$ is an ideal, then}
\[
\topdim(B)=\max\{\topdim(I),\topdim(B/I)\}.
\]

We prove the result for $d$-subhomogeneous \ca{s} by induction over $d$.
First, note that a \ca{} is $1$-subhomogeneous if and only if it is $1$-homogeneous (if and only if it is commutative).
In this case, the result follows directly from Fact~1.

Next, let $d\geq 1$ and assume that the result holds for every $d$-subhomogeneous \ca.
Let $A$ be $(d+1)$-subhomogeneous.
We need to show that the set $\mathcal{S} := \{ B\in\SubSep(A) : \topdim(B)\leq n \}$ is $\sigma$-complete and cofinal.

To verify that $\mathcal{S}$ is $\sigma$-complete, let $\mathcal{T}\subseteq\mathcal{S}$ be a countable, directed family.
Set $C:=\overline{\bigcup\{B:B\in\mathcal{T}\}}$.
Then $C$ is a separable \ca{} that is approximated by the sub-\ca{s} $B\subseteq C$ for $B\in\mathcal{T}$ that each satisfy $\topdim(B)\leq n$.
By \cite[Proposition~8]{Thi13TopDimTypeI}, we have $\topdim(C)\leq n$.
Thus, $C\in\mathcal{S}$, as desired.

Next, we verify that $\mathcal{S}$ is cofinal.
Set $I:=I_{\geq d+1}\subseteq A$, the ideal of all elements in~$A$ that vanish under all irreducible representations of dimension at most $d$.
Then $I$ is $(d+1)$-homogeneous and $A/I$ is $d$-subhomogeneous.
By Fact~2, we have $\topdim(I)\leq n$ and $\topdim(A/I)\leq n$.
By Fact~1 and by the assumption of the induction, the collections
\begin{align*}
\mathcal{T}_1 &:= \big\{ C\in\SubSep(I) : \topdim(C)\leq n \big\}, \\
\mathcal{T}_2 &:= \big\{ D\in\SubSep(A/I) : \topdim(D)\leq n \big\}, 
\end{align*}
are $\sigma$-complete and cofinal.
By \cite[Lemma~3.2]{Thi20arX:grSubhom}, it follows that the families
\begin{align*}
\mathcal{S}_1 &:= \big\{ B\in\SubSep(A) : \topdim(B\cap I)\leq n \big\}, \\
\mathcal{S}_2 &:= \big\{ B\in\SubSep(A) : \topdim(B/(B\cap I))\leq n \big\}, 
\end{align*}
are $\sigma$-complete and cofinal.
Then $\mathcal{S}_1\cap\mathcal{S}_2$ is $\sigma$-complete and cofinal as well.
Given $B\in\mathcal{S}_1\cap\mathcal{S}_2$, it follows from Fact~2 that
\[
\topdim(B) = \max\{\topdim(B\cap I),\topdim(B/(B\cap I))\} \leq n.
\]
Thus, $\mathcal{S}_1\cap\mathcal{S}_2\subseteq\mathcal{S}$.
Since $\mathcal{S}_1\cap\mathcal{S}_2$ is cofinal, so is $\mathcal{S}$.
\end{proof}

%==========================================================================================
We deduce a result that is probably known to the experts, but which does not appear in the literature so far.
%Here $\dr(A)$ denote the decomposition rank of a \ca{} $A$.
The equality of the topological dimension and the decomposition rank $\dr(A)$ of a \emph{separable} subhomogeneous \ca{} $A$ was shown in \cite{Win04DRSubhomogeneous}.

%==========================================================================================
\begin{thm}
\label{prp:dimnucSH}
Let $A$ be a subhomogeneous \ca{}.
Then
\[
\dimnuc(A)=\dr(A)=\topdim(A).
\]
\end{thm}
\begin{proof}
As noted in \cite[Remarks~2.2(ii)]{WinZac10NuclDim}, the inequality $\dimnuc(B)\leq\dr(B)$ holds for every \ca{} $B$.
To verify $\dr(A)\leq\topdim(A)$, set $n:=\topdim(A)$.
We may assume that $n$ is finite.
By \autoref{prp:SubSepSH}, the family
\[
\mathcal{S} := \big\{ B\in\SubSep(A) : \topdim(B)\leq n \big\}
\]
is cofinal.
Each $B\in\mathcal{S}$ is a separable, subhomogeneous \ca{}, whence we can apply \cite[Theorem~1.6]{Win04DRSubhomogeneous} to deduce that $\dr(B)=\topdim(B)\leq n$.
Thus, $A$ is approximated by the collection $\mathcal{S}$ consisting of \ca{s} with decomposition rank at most $n$.
It is straightforward to verify that this implies $\dr(A)\leq n$.

To verify $\topdim(A)\leq\dimnuc(A)$, set $m:=\dimnuc(A)$, which we may assume to be finite.
It follows from \cite[Proposition~2.6]{WinZac10NuclDim} that the family
\[
\mathcal{T} := \big\{ B\in\SubSep(A) : \dimnuc(B)\leq m \big\}
\]
is cofinal.
Let $B\in\mathcal{T}$.
Then $B$ is a separable, subhomogeneous \ca{}.
For each $k\geq 1$, let $B_k$ be the canonical $k$-homogeneous ideal-quotient of $B$;
see \autoref{pgr:SH}.
Using \cite[Corollary~2.10]{WinZac10NuclDim} at the first step, and using that the nuclear dimension does not increase when passing to ideals (\cite[Proposition~2.5]{WinZac10NuclDim}) or quotients (\cite[Proposition~2.3(iv)]{WinZac10NuclDim}) at the second step, we get
\[
\topdim(B_k)
= \dimnuc(B_k)
\leq \dimnuc(B) \leq m.
\]

Using that $B$ is obtained as a successive extension of $B_1$ by $B_2$, and then by $B_3$, and so on, it follows from \cite[Proposition~2.2]{BroPed09Limits} (see Fact~2 in the proof of \autoref{prp:SubSepSH}) that $\topdim(B)\leq m$.
Thus, $A$ is approximated by the collection~$\mathcal{T}$ consisting of \ca{s} with topological dimension at most $m$.
By \cite[Proposition~8]{Thi13TopDimTypeI}, we get $\topdim(A)\leq m$, as desired.
\end{proof}

%==========================================================================================
\begin{lma}
\label{prp:CommutativeEstimate}
Let $X$ be a compact, Hausdorff space.
Then
\[
\dim(X)\leq\dim(\Cu(C(X))).
\]
\end{lma}
\begin{proof}
Set $n:=\dim(\Cu(C(X)))$, which we may assume to be finite.
To verify that $\dim(X)\leq n$, let $\mathcal{U}=\{U_{1},\ldots,U_{r}\}$ be a finite open cover of $X$.
We need to find a $(n+1)$-colourable, finite, open refinement of $\mathcal{U}$;
see \autoref{rmk:dim}.

Since $X$ is a normal space, we can find an open cover $\mathcal{V}=\{V_{1},\ldots,V_{r}\}$ of $X$ such that $\overline{V_j}\subseteq U_j$ for each $j$;
see for example \cite[Proposition~1.3.9, p.20]{Pea75DimThy}.
For each $j$, by Urysohn's lemma we obtain a continuous function $f_j\colon X\to[0,1]$ that takes the value $1$ on $\overline{V_j}$ and that vanishes on $X\setminus U_j$.

We have $1 \leq f_1+\ldots+f_r$, and therefore
\[
[1] 
\ll [1] 
\leq [f_1+\ldots+f_r]
\leq [f_1]+\ldots+[f_r]
\]
in $\Cu(C(X))$.
Using that $\dim(\Cu(C(X)))\leq n$, we obtain elements $z_{j,k}\in\Cu(C(X))$ for $j=1,\ldots,r$ and $k=0,\ldots,n$ satisfying (i)-(iii) in \autoref{dfn:dim}.

For each $j$ and $k$, choose $g_{j,k}\in(C(X)\otimes\KK)_+$ such that $z_{j,k}=[g_{j,k}]$.
Viewing $g_{j,k}$ as a positive, continuous function $g_{j,k}\colon X\to\KK$, we set
\[
W_{j,k}:=\{x\in X : g_{j,k}(x)\neq 0 \}.
\]
Then $W_{j,k}$ is an open set.
Condition~(i) implies that $g_{j,k}=\lim_n h_nf_jh_n^*$ for some sequence $(h_n)_n$ in $C(X)\otimes\KK$.
Thus, $g_{j,k}(x)=0$ whenever $f_j(x)=0$, which shows that $W_{j,k}\subseteq U_j$.
Condition~(ii) implies that $X$ is covered by the sets $W_{j,k}$.
Thus, the family $\mathcal{W}:=\{W_{j,k}\}$ is a finite, open refinement of $\mathcal{U}$.

Let $k\in\{0,\ldots,n\}$.
Given $x\in X$, it follows from condition~(iii) that the rank of $g_{1,k}(x)\oplus\ldots\oplus g_{r,k}(x)$ is at most one.
This implies that at most one of $g_{1,k}(x),\ldots,g_{r,k}(x)$ is nonzero.
Thus, the sets $W_{1,k},\ldots,W_{r,k}$ are pairwise disjoint.

Hence, $\mathcal{W}$ is $(n+1)$-colourable, as desired.
%For each $k$, condition~(iii) implies that the sets $V_{1,k},\ldots,V_{r,k}$ are pairwise disjoint.
%We claim that the sets $W_{1,k},\ldots,W_{r,k}$ are pairwise disjoint.
%To prove this, assume that there exists $x\in W_{i,k}\cap W_{j,k}$ for some $i\neq j$.
%Then $g_{i,k}(x)\neq 0$ and $g_{j,k}(x)\neq 0$.
%It follows that the rank of the compact operator $g_{i,k}(x)\oplus g_{j,k}(x)$ is at least two, which is i
\end{proof}

%==========================================================================================
\begin{prp}
\label{prp:CommutativeUnital}
Let $X$ be a compact, Hausdorff space.
Then
\[
\dim(\Cu(C(X))) = \dim(X).
\]
\end{prp}
\begin{proof}
The inequality `$\geq$' is shown in \autoref{prp:CommutativeEstimate}.
By \cite[Proposition~2.4]{WinZac10NuclDim}, we have $\dim(X)=\dimnuc(C(X))$ if $X$ is second-countable.
By \autoref{prp:dimnucSH}, this  also holds for arbitrary compact, Hausdorff spaces.
Thus, the inequality `$\leq$' follows from \autoref{prp:Nuclear_bound}.
\end{proof}

%==========================================================================================
\begin{cor}
\label{cor:Lsc}
Let $X$ be a compact, metrizable space. 
Then
\[
\dim( \Lsc(X,\NNbar) ) = \dim(X).
\]
\end{cor}
\begin{proof}
It is enough to see that $\Lsc (X,\NNbar )$ is a retract of $\Cu (C(X))$, since the inequality '$\geq $' has already been proven in \autoref{exa:Lsc} and the inequality '$\leq $' will follow from \autoref{prp:DimRetract} and \autoref{prp:CommutativeUnital}.

Thus, set $S= \Lsc (X,\NNbar )$ and $T=\Cu (C(X))$. Define $\iota\colon \Lsc (X,\NNbar )\to \Cu (C(X))$ as the unique \CuMor{} mapping the characteristic function $\chi_{U}$ to
the class of a positive function in $C(X)$ with support $U$ for every open subset $U\subset X$.

Also, let $\sigma\colon T\to S$ be the generalized \CuMor{} mapping the class of an element $a\in C(X)\otimes\mathcal{K}$ to its rank function $\sigma (a)\colon X\to\NNbar$, $\sigma (a)(x)= \text{rank} (a(x))$.

It is easy to check that $\sigma\circ \iota =\id_S$, as desired.
\end{proof}

%==========================================================================================
\begin{thm}
\label{prp:Commutative}
Let $X$ be a locally compact, Hausdorff space.
Then
\[
\dim(\Cu(C_0(X))) = \locdim(X).
\]
\end{thm}
\begin{proof}
Let $K\subseteq X$ be a compact subset.
Then $C(K)$ is a quotient of $C_0(X)$.
Using \autoref{prp:CommutativeUnital} at the first step and \autoref{prp:PermanenceDimCu} at the second step, we get
\[
\dim(K) = \dim(\Cu(C(K))) \leq \dim(\Cu(C_0(X))).
\]
It follows that $\locdim(X)\leq\dim(\Cu(C_0(X)))$.

Conversely, we use that $C_0(X)$ is an ideal in $C(\alpha X)$.
Applying \autoref{prp:PermanenceDimCu} at the first step, and using \autoref{prp:CommutativeUnital} and $\dim(\alpha X)=\locdim(X)$ at the second step, we get
\[
\dim(\Cu(C_0(X)))
\leq \dim(\Cu(C(\alpha X)))
= \locdim(X).
\]
This show the converse inequality and finishes the proof.
\end{proof}

%==========================================================================================
\begin{lma}
\label{prp:H}
Let $A$ be a homogeneous \ca.
Then
\[
\dimnuc(A) \leq \dim(\Cu(A)).
\]
\end{lma}
\begin{proof}
Let $d\geq 1$ be such that $A$ is $d$-homogeneous.
Set $X:=\Prim(A)$, which is locally compact and Hausdorff.
Then $\topdim(A)=\locdim(X)$, and we need to show that $\locdim(X)\leq\dim(\Cu(A))$.

Let $x\in X$.
Since $A$ is the algebra of sections vanishing at infinity of a locally trivial $M_d(\CC)$-bundle over $X$, there exists a compact neighborhood $Y$ of $x$ over which the bundle is trivial.
Let $I\subseteq A$ be the ideal of all sections in $A$ that vanish on $X\setminus Y$.
Then $A/I$ is the algebra of sections of the trivial $M_d(\CC)$-bundle over $Y$, and so $A/I\cong C(Y)\otimes M_d$.
Using \autoref{prp:CommutativeEstimate} at the first step, using that $C(Y)$ and $C(Y)\otimes M_d$ have isomorphic Cuntz semigroups at the second step, and using \autoref{prp:PermanenceDimCu} at the last step, we get
\[
\dim(Y) \leq \dim(\Cu(C(Y))) = \dim(\Cu(C(Y)\otimes M_d)) \leq \dim(\Cu(A)).
\]
Thus, every point in $X$ has a closed neighborhood of dimension at most $\dim(\Cu(A))$, whence $\locdim(X)\leq\dim(\Cu(A))$, as desired.
\end{proof}

%==========================================================================================
\begin{thm}
\label{prp:SH}
Let $A$ be a subhomogeneous \ca.
Then
\[
\dim(\Cu(A)) = \dimnuc(A) = \dr(A) = \topdim(A).
\]
\end{thm}
\begin{proof}
The second and third equalities are shown in \autoref{prp:dimnucSH}.
By \autoref{prp:Nuclear_bound}, the inequality $\dim(\Cu(A)) \leq\dimnuc(A)$ holds in general.
It remains to verify that $\topdim(A)\leq\dim(\Cu(A))$.

For each $k\geq 1$, let $A_k$ be the canonical $k$-homogeneous ideal-quotient of $A$ as in \autoref{pgr:SH}.
Using \autoref{prp:H} at the first step, and using \autoref{prp:PermanenceDimCu} at the second step, we get 
\[
\topdim(A_k)
\leq \dim(\Cu( A_k ))
\leq \dim(\Cu(A)).
\]
Consequently,
\[
\topdim(A) = \max_{k\geq 1} \topdim(A_k)
\leq \dim(\Cu(A)),
\]
as desired.
\end{proof}

%==========================================================================================
\begin{exa}
\label{exa:DimSmallerDimNuc}
There are many examples showing that \autoref{prp:SH} does not hold for all \ca{s}.
In \autoref{prp:CharDim0UnitalSR1}, we will show that every \ca{} $A$ of real rank zero satisfies $\dim(\Cu(A))=0$.
On the other hand, a separable \ca{} $A$ satisfies $\dimnuc(A)=0$ if and only if $A$ is an AF-algebra;
see \cite[Remarks~2.2(iii)]{WinZac10NuclDim}.
Thus, every separable \ca{} $A$ of real rank zero that is not an AF-algebra is an example where $\dim(\Cu(A))$ is strictly smaller than $\dimnuc(A)$.
More extremely, every non-nuclear \ca{} $A$ of real rank zero, such as $\Bdd(\ell^2(\NN))$, satisfies $\dim(\Cu(A))=0$ while $\dimnuc(A)=\infty$.
Another example is the irrational rotation algebra $A_\theta$, which satisfies $\dim(\Cu(A_\theta))=0$ while $\dimnuc(A_\theta)=1$.
\end{exa}

%==========================================================================================
%==========================================================================================
\section{Algebraic, zero-dimensional Cuntz semigroups}
\label{sec:algebraic}

%==========================================================================================
In this section we begin our systematic study of zero-dimensional \CuSgp{s}.
After giving a useful characterization of zero-dimensionality (\autoref{prp:dim0_reduction}), we provide a sufficient criterion: A \CuSgp{} is zero-dimensional whenever it contains a sup-dense subsemigroup that satisfies the Riesz decomposition property with respect to the pre-order induced by the way-below relation;
see \autoref{prp:denseRiesz}.
We deduce that the Cuntz semigroup of every \ca{} of real rank zero is zero-dimensional; 
and conversely, every unital \ca{} of stable rank one and with zero-dimensional Cuntz semigroup has real rank zero;
\autoref{prp:CharDim0UnitalSR1}.

We also show that every weakly cancellative, zero-dimensional \CuSgp{} satisfying \axiomO{5} contains a largest algebraic ideal, which contains all compact elements;
see \autoref{prp:LargestAlgebraicIdeal}.
In \autoref{sec:simple}, we study certain zero-dimensional \CuSgp{s} that contain no compact elements.

%==========================================================================================
\begin{lma}
\label{prp:dim0_reduction}
Let $S$ be a \CuSgp.
Then $\dim(S)=0$ if and only if, whenever $x'\ll x\ll y_{1}+y_{2}$ in $S$, there exist $z_{1},z_2\in S$ such that
\[
z_1\ll y_{1}, \quad
z_{2}\ll y_{2}, \andSep 
x'\ll z_{1}+z_{2}\ll x.
\]
\end{lma}
\begin{proof}
The forward implication is clear, so we are left to prove the converse.
Given $r\geq 1$ and $x'\ll x\ll y_{1}+\ldots+y_{r}$ in $S$, we need to find $z_1,\ldots,z_r\in S$ such that
\[
z_j \ll y_j \ \text{ for } j=1,\ldots,r, 
\andSep
x'\ll z_{1}+\ldots+z_{r}\ll x.
\]
We prove this by induction on $r$.
The case $r=1$ is clear and the case $r=2$ holds by assumption.
 
Thus, let $r>2$ and assume that the result holds for $r-1$.
Given $x'\ll x\ll y_{1}+\ldots+y_{r}$, apply the assumption to
\[
x'\ll x \ll (y_{1}+\ldots+y_{r-1})+y_r
\]
to obtain $u_1,u_2\in S$ such that
\[
u_1 \ll y_{1}+\ldots+y_{r-1}, \quad
u_2 \ll y_r, \andSep
x' \ll u_1+u_2 \ll x.
\]

Choose $u_1'$ such that
$
u_1'\ll u_1$ and $x'\ll u_1'+u_2$. Applying the induction hypothesis to
\[
u'_{1}\ll u_{1}\ll y_{1}+\ldots+y_{r-1},
\]
we obtain $z_1,\ldots,z_{r-1}\in S$ such that
\[
z_j \ll y_j \ \text{ for } j=1,\ldots,r-1, 
\andSep
u_1'\ll z_{1}+\ldots+z_{r-1}\ll u_1.
\]
Set $z_r:=u_2$.
Then $z_1,\ldots,z_r$ have the desired properties.
\end{proof}

%==========================================================================================
\begin{rmk}
It follows from \autoref{prp:dim0_reduction} that every zero-dimensional \CuSgp{} satisfies \axiomO{6}.
The converse does not hold, that is, zero-dimensionality is strictly stronger than \axiomO{6}.
For example, the Cuntz semigroup of the Jiang-Su algebra satisfies \axiomO{6} but is not zero-dimensional;
see \autoref{exa:Z}.
\end{rmk}

%==========================================================================================
Recall that a semigroup $S$ with a pre-order $\prec$ is said to satisfy the \emph{Riesz decomposition property} if whenever $x,y,z\in S$ satisfy $x\prec y+z$, then there exist $e,f\in S$ such that $x=e+f$, $e\prec y$ and $f\prec z$.

%==========================================================================================
\begin{prp}
\label{prp:denseRiesz}
Let $S$ be a \CuSgp{}, and let $D\subseteq S$ be a sup-dense subsemigroup such that $D$ satisfies the Riesz decomposition property for the pre-order induced by $\ll$.
Then $\dim(S)=0$.
\end{prp}
\begin{proof}
To verify the condition in \autoref{prp:dim0_reduction}, let $x'\ll x\ll y_{1}+y_{2}$ in $S$.
Using that $D$ is sup-dense, we find $\tilde{x},\tilde{y}_1 ,\tilde{y}_2\in D$ such that
\[
x' \ll \tilde{x} \ll x \leq \tilde{y}_1+\tilde{y}_2, \quad
\tilde{y}_1\ll y_1, \andSep
\tilde{y}_2\ll y_2.
\]
Then $\tilde{x} \ll \tilde{y}_1+\tilde{y}_2$.
Using that $D$ satisfies the Riesz decomposition property, we obtain $x_1,x_2\in D$ such that
\[
\tilde{x} = x_1+x_2, \quad x_1\ll\tilde{y}_1, \andSep
x_2\ll\tilde{y}_2.
\]
Then $x_1$ and $x_2$ have the desired properties to verify the condition of \autoref{prp:dim0_reduction}.
\end{proof}

%==========================================================================================
Recall that a \CuSgp{} is said to be \emph{algebraic} if its compact elements are sup-dense;
see \cite[Section~5.5]{AntPerThi18TensorProdCu}.

%==========================================================================================
\begin{lma}
\label{prp:AlgebraicIdeal}
Let $S$ be a weakly cancellative \CuSgp{} satisfying \axiomO{5} and $\dim(S)=0$.
Let $c\in S$ be compact.
Then the ideal generated by $c$ is algebraic.
\end{lma}
\begin{proof}
Let $I$ be the ideal generated by $c$.
Note that $x\in S$ belongs to $I$ if and only if $x\leq\infty c$.
To verify that $I$ is algebraic, let $x',x\in I$ satisfy $x'\ll x$.
We need to find a compact element $z$ such that $x'\ll z\ll x$. 

Choose $x''\in S$ such that $x'\ll x''\ll x$.
Then $x''\ll x\leq\infty c$, which allows us to choose $n\in\NN$ such that $x''\leq nc$.
Applying \axiomO{5} to $x'\ll x''\leq nc$, we obtain $y\in S$ such that 
\[
x'+y\leq nc \leq x''+y.
\]
 
Using that $\dim(S)=0$ for $nc\ll nc\ll x''+y$, we obtain $z_1,z_2\in S$ such that
\[
nc = z_1+z_2, \quad z_1\ll x'', \andSep z_2\ll y.
\]

By weak cancellation, $z_1$ and $z_2$ are compact.
We now have
\[
x'+y\ll nc = z_{1}+z_{2} \leq z_1 + y.
\]
Using weak cancellation, we get $x'\ll z_{1}$.
Thus, $z_1$ has the desired properties. 
\end{proof}

%==========================================================================================
\begin{prp}
\label{prp:LargestAlgebraicIdeal}
Let $S$ be a weakly cancellative \CuSgp{} satisfying \axiomO{5} and $\dim(S)=0$.
Then $S$ contains a largest algebraic ideal, which agrees with the ideal generated by all compact elements of $S$.
\end{prp}
\begin{proof}
This follows directly from \autoref{prp:AlgebraicIdeal}.
\end{proof}

%==========================================================================================
\begin{prp}
\label{prp:charAlgDim0}
Let $S$ be a weakly cancellative \CuSgp{} satisfying \axiomO{5}.
Then the following are equivalent:
\begin{enumerate}
\item
We have $\dim(S)=0$, and the set of compact elements of $S$ is full (that is, there is no proper ideal of $S$ containing all compact elements);
\item
$S$ is algebraic and satisfies \axiomO{6}.
\end{enumerate}
\end{prp}
\begin{proof}
Assuming~(1), it follows from \autoref{prp:AlgebraicIdeal} that $S$ is algebraic.
Further, it is clear that $\dim(S)=0$ implies that $S$ satisfies \axiomO{6}.

Conversely, assuming~(2), set $D:=\{x\in S:x\ll x\}$, the semigroup of compact elements.
By assumption, $D$ is sup-dense.
By \cite[Corollary~5.5.10]{AntPerThi18TensorProdCu}, $D$ satisfies the Riesz decomposition property.
Hence, $\dim(S)=0$ by \autoref{prp:denseRiesz}.

Since $S$ is algebraic, it is clear that compact elements of $S$ are full.
\end{proof}

%==========================================================================================
\begin{thm}
\label{prp:CharDim0UnitalSR1}
If $A$ is a \ca{} of real rank zero, then $\dim(\Cu(A))=0$.
Conversely, if $A$ is a unital \ca{} of stable rank one, then $A$ has real rank zero if (and only if) $\dim(\Cu(A))=0$.
\end{thm}
\begin{proof}
1. Let $A$ be a \ca{} of real rank zero.
It follows that the submonoid $\Cu(A)_c$ of compact elements in $\Cu(A)$ is sup-dense.
With view towards \autoref{prp:denseRiesz}, it suffices to show that $\Cu(A)_c$ satisfies the Riesz decomposition property.

By \cite[Corollary~3.3]{BroPed91CAlgRR0}, $A\otimes\KK$ has real rank zero.
Hence, it follows from \cite[Theorem~1.1]{Zha90RieszDecomp} that the Murray-von Neumann semigroup $V(A)$ of equivalence classes of projections in $A\otimes\KK$ satisfies the Riesz decomposition property.
Given a projection $p\in A\otimes\KK$, we let $[p]_0$ denote its equivalence class in $V(A)$.
Then the map $V(A)\to\Cu(A)_c$, given by $[p]_0\mapsto[p]$, is well-defined, additive and order-preserving.
Using that $A$ has real rank zero, it follows that $\alpha$ is surjective onto $\Cu(A)_c$.
(However, $\alpha$ is not injective in general, and $V(A)$ and $\Cu(A)_c$ need not be isomorphic.)

To verify that $\Cu(A)_c$ satisfies the Riesz decomposition property, let $p,q,r$ be projections in $A\otimes\KK$ such that $[p]\ll[q]+[r]$ in $\Cu(A)$.
Then $[p]_0\leq[q]_0+[r]_0$ in $V(A)$.
Using that $V(A)$ satisfies the Riesz decomposition property, we obtain projections $q'\leq q$ and $r'\leq r$ such that $[p]_0=[q']_0+[r']_0$ in $V(A)$.
It follows that $[p'],[q']\in\Cu(A)_c$ satisfy $[p]=[q']+[r']$, $[q']\ll[q]$ and $[r']\ll[r]$, as desired.

2. Let $A$ be a unital \ca{} of stable rank one and assume $\dim(\Cu(A))=0$.
Since $A$ is unital, the compact elements in $\Cu(A)$ form a full subset.
Thus, by \autoref{prp:charAlgDim0}, $\Cu(A)$ is algebraic.
Now it follow from \cite[Corollary 5]{CowEllIva08CuInv} that~$A$ has real rank zero.
\end{proof}

%==========================================================================================
\begin{cor}
\label{prp:simple_Zstable}
Let $A$ be a separable, simple, $\mathcal{Z}$-stable \ca. 
Then we have $\dim(\Cu(A))\leq 1$.
Moreover, $\dim(\Cu(A))=0$ if and only if $A$ has real rank zero or if $A$ is stably projectionless.
\end{cor}
\begin{proof}
It follows from \cite[Theorem~4.1.10]{Ror02Classification} that $A$ is either purely infinite or stably finite.
%Simple, purely infinite \ca{s} contain nonzero projections and are therefore never stably projectionless.
Thus, we can distinguish three cases:
$A$ is either purely infinite or stably projectionless, or stably finite and not stably projectionless.

The first statement follows from \autoref{prp:ZstableCAlg}.
To show the forward implication of the second statement, assume that $\dim(\Cu(A))=0$.
We need to show that $A$ has real rank zero or is stably projectionless.
First, if $A$ is purely infinite, then $A$ has real rank zero;
see \cite[Proposition~V.3.2.12]{Bla06OpAlgs}.
Second, if $A$ is stably projectionless, then there is nothing to show.
Third, we consider the case that $A$ is stably finite and not stably projectionless.
Let $p\in A\otimes\KK$ be a nonzero projection.
Then $p(A\otimes\KK)p$ is a separable, unital, simple, stably finite, $\mathcal{Z}$-stable \ca{} and therefore has stable rank one by \cite[Theorem~6.7]{Ror04StableRealRankZ}.
Since $A$ and $p(A\otimes\KK)p$ are stably isomorphic, they have isomorphic Cuntz semigroups.
Thus, $\dim(\Cu(p(A\otimes\KK)p))=0$, and we deduce from \autoref{prp:CharDim0UnitalSR1} that $p(A\otimes\KK)p$ has real rank zero.
By \cite[Corollary~2.8 and~3.3]{BroPed91CAlgRR0}, a \ca{} has real rank zero if and only if its stabilization does.
Thus, $A$ has real rank zero.

To show the backward implication of the second statement, assume that $A$ has real rank zero or is stably projectionless.
We need to show that $\dim(\Cu(A))=0$.
If $A$ has real rank zero, this follow from \autoref{prp:CharDim0UnitalSR1}.
Let us consider the case that $A$ is stably projectionless.
Then $\Cu(A)$ contains no nonzero compact elements by \cite{BroCiu09IsoHilbModSF}.
Thus, $\Cu(A)$ is soft and has $Z$-multiplication, which by \cite[Theorem~7.5.4]{AntPerThi18TensorProdCu} implies that $\Cu(A)$ has $[0,\infty]$-multiplication.
Hence, $\dim(\Cu(A))=0$ by \autoref{prp:RZ_mult}.
\end{proof}

%==========================================================================================
%==========================================================================================
\section{Thin boundary and complementable elements}
\label{sec:thin}

%==========================================================================================
In this section, we study soft elements in simple \CuSgp{s} that behave very similar to compact elements:
the elements with thin boundary (\autoref{dfn:ThinBoundary}), and the complementable elements (\autoref{dfn:complementable}).
If $S$ is a simple, stably finite, soft \CuSgp{} satisfying \axiomO{5} and \axiomO{6} (for example, the Cuntz semigroup of a simple, stably projectionless \ca{}; see \autoref{prp:projectionless}), then every element with thin boundary is complementable;
see \autoref{prp:ThinImplComplementable}.
The converse holds if $S$ is also weakly cancellative (for example, the Cuntz semigroup of a simple, stably projectionless \ca{} of stable rank one);
see \autoref{prp:ThinIffComplementable}.

In \autoref{sec:simple}, we will show that zero-dimensionality of certain simple \CuSgp{s} is characterized by sup-denseness of the elements with thin boundary.

%==========================================================================================
\begin{pgr}
\label{pgr:soft}
We say that a simple \CuSgp{} $S$ is \emph{stably finite} if for all $x,z\in S$, we have that $x+z\ll z$ implies $x=0$.
Using that $S$ is simple, one can show that this definition is equivalent to the one given in \cite[Paragraph~5.2.2]{AntPerThi18TensorProdCu}.
We note that every simple, weakly cancellative \CuSgp{} is stably finite.

Let $S$ be a simple, stably finite \CuSgp{} satisfying \axiomO{5}.
Recall that an element $x\in S$ is \emph{compact} if $x\ll x$.
We say that $x\in S$ is \emph{soft} if $x=0$ or if $x\neq 0$ and for every $x'\in S$ satisfying $x'\ll x$ there exists a nonzero $t\in S$ such that $x'+t\ll x$.
(Using \cite[Proposition~5.3.8]{AntPerThi18TensorProdCu}, one sees that this is equivalent to the original definition.)
We say that $S$ is \emph{soft} if every element in $S$ is soft.

We let $S_c$ and $S_{\rm{soft}}$ denote the set of compact and soft elements in $S$, respectively.
We also set $S_{\rm{soft}}^{\times} := S_{\rm{soft}} \setminus\{0\}$.
It is easy to see that $S_c, S_{\rm{soft}}$ and $S_{\rm{soft}}^{\times}$ are submonoids of $S$.
Further, $S_{\rm{soft}}^{\times}$ is absorbing in the sense that $x+y$ belongs to $S_{\rm{soft}}^{\times}$ whenever $x$ or $y$ does;
see \cite[Theorem~5.3.11]{AntPerThi18TensorProdCu}.

By \cite[Proposition~5.3.16]{AntPerThi18TensorProdCu}, every element in $S$ is either compact, or nonzero and soft.
Hence, $S$ can be decomposed as $S=S_{\rm{soft}}^{\times}\sqcup S_{c}$.
\end{pgr}

%==========================================================================================
\begin{prp}
\label{prp:projectionless}
Let $A$ be a simple, stably projectionless \ca.
Then $\Cu(A)$ is a simple, stably finite, soft \CuSgp{} satisfying \axiomO{5} and \axiomO{6}.
\end{prp}
\begin{proof}
The Cuntz semigroup $\Cu (A)$ is simple and satisfies \axiomO{5} and \axiomO{6} since it is the Cuntz semigroup of a simple \ca{} (see, for example, \cite[Corollary~5.1.12]{AntPerThi18TensorProdCu}). As $A$ is stably projectionless, $\Cu (A)$ has no nonzero compact elements by \cite{BroCiu09IsoHilbModSF}.

It is easy to check that a simple $\Cu$-semigroup is stably finite if and only if $\infty$ is not compact or if $S$ is zero. Therefore, the Cuntz semigroup of a stably projectionless \ca{} is always stably finite. 

By \cite[Proposition~5.3.16]{AntPerThi18TensorProdCu} we have $\Cu (A)^{\times}=\Cu (A)_{\rm{soft}}^{\times}$ as desired.
\end{proof}

%==========================================================================================
\begin{dfn}
\label{dfn:ThinBoundary}
Let $S$ be a simple \CuSgp.
We say that an element $x\in S$ has \emph{thin boundary} if $x\ll x+t$ for every nonzero $t\in S$.
We let $S_\thin$ denote the set of elements in $S$ with thin boundary.
\end{dfn}

%==========================================================================================
Note that every compact element has thin boundary, but the converse is not true:
In $[0,\infty]$ every element has thin boundary, but only $0$ is compact.

%==========================================================================================
\begin{exa}
\label{exa:LAff}
Let $K$ be a metrizable, compact, convex set.
Let $\LAff(K)_{++}$ denote the set of lower semicontinuous, affine functions $K\to(0,\infty]$.
Equipped with pointwise order and addition, $S:=\LAff(K)_{++}\cup\{0\}$ is a simple \CuSgp;
see, for example, \cite[Proposition~3.9]{Thi20RksOps}.
By \cite[Lemma~3.6]{Thi20RksOps}, $f,g\in S$ satisfy $f\ll g$ if and only if there exist $\varepsilon>0$ and a continuous, finite-valued function $h\in S$ such that $f+\varepsilon\leq h\leq g$.
Using this, we deduce that $f\in S$ has thin boundary in the sense of \autoref{dfn:ThinBoundary} if and only if $f$ is continuous and finite-valued.
\end{exa}

%==========================================================================================
\begin{exa}
\label{exa:ContinuousRank}
Let $S$ be a countably based, simple, stably finite \CuSgp{} containing a nonzero, compact element $u\in S$.
Let $K$ denote the metrizable, compact, convex set of functionals $\lambda\colon S\to[0,\infty]$ satisfying $\lambda(u)=1$.
The rank of an element $x\in S$ is the function $\widehat{x}\colon K\to[0,\infty]$ given by $\widehat{x}(\lambda):=\lambda(x)$.

Given $x\in S$, its rank $\widehat{x}$ belongs to the simple \CuSgp{} $\LAff(K)_{++}\cup\{0\}$.
If $x$ has thin boundary, then $\widehat{x}$ is continuous (using for example \cite[Lemma~2.2.5]{Rob13Cone}) and finite-valued, and it follows from \autoref{exa:LAff} that $\widehat{x}\in \LAff(K)_{++}\cup\{0\}$ has thin boundary as well.
%The converse does not hold in general.
%Indeed, if $K$ is a singleton, then $\widehat{x}$ is always continuous, but ... ??
\end{exa}

%==========================================================================================
\begin{rmk}
\label{rmk:ThinBoundary}
The definition of `thin boundary' is inspired by the notion of `small boundary' in dynamical systems.
Let $T\colon X\to X$ be a minimal homeomorphism on a compact, metrizable space $X$.
We let $M_T(X)$ denote the set of $T$-invariant probability measures on $X$, which is a non-empty, metrizable Choquet simplex.

The rank of an open subset $U\subseteq X$ is the map $\widehat{U}\colon M_T(X)\to[0,1]$ given by $\widehat{U}(\mu):=\mu(U)$.
Then $\widehat{U}$ belongs to the simple \CuSgp{} $\LAff(M_T(X))_{++}\cup\{0\}$.
%is lower semicontinuous and affine.
%Since $T$ is minimal, $\widehat{U}$ is strictly positive whenever $U\neq\emptyset$.

An open set $U\subseteq X$ is said to have `small boundary' if $\mu(\partial U)=0$ for every $\mu\in M_T(X)$;
see \cite[Section~3]{Lin99MeanDim}.
If $U$ has small boundary, then $\widehat{U}$ is continuous.
Indeed, consider the open set $V:=X\setminus\overline{U}$.
Then $\widehat{U}$ and $\widehat{V}$ are lower semicontinuous functions, which by assumption add to the constant function $1$, which implies that they are continuous.
Thus, just as `thin boundary' implies continuous rank (\autoref{exa:ContinuousRank}), so does `small boundary'.

Further, to the dynamical system $(X,T)$ one can associate a dynamical version of the Cuntz semigroup based on the dynamical notion of comparison defined in \cite[Section~3]{Ker20DimCompAlmFin}.
One can then show that an open set $U\subseteq X$ has small boundary whenever it has `thin boundary' in the sense that $[U]\ll[U]+[V]$ in the dynamical Cuntz semigroup for every nonempty open set $V\subseteq X$.

Thus, in the dynamical setting, `thin boundary' implies `small boundary', and we can think of `small boundary' as the measurable (or tracial) version of `thin boundary'.
%Further, one can consider a dynamical version of the Cuntz semigroup associated to the dynamical system $(X,T)$ with a dynamic notion of comparison defined as follows:
%given open subsets $U,V\subseteq X$, we set $U\precsim V$ if for every closed subset $C\subseteq U$ there exist open sets $W_1,\ldots,W_n$ that cover $C$ and integers $k_1,\ldots,k_n$ such that $T^{k_1}(W_1),\ldots,T^{k_n}(W_n)$ are pairwise disjoint and contained in $V$;
%see \cite[Section~3]{Ker20DimCompAlmFin}.
%In the associated pre $[U]=\sup_n[U_n]$
\end{rmk}

%==========================================================================================
We will repeatedly use the following result.

%==========================================================================================
\begin{lma}
\label{prp:smallElements}
Let $S$ be a simple, nonelementary \CuSgp{} satisfying \axiomO{5} and \axiomO{6}.
Let $u_0,u_1\in S$ be nonzero.
Then there exists a nonzero $w\in S$ such that $2w\ll u_0,u_1$.
\end{lma}
\begin{proof}
This follows by combining \cite[Lemma~5.1.18]{AntPerThi18TensorProdCu} and \cite[Proposition~5.2.1]{Rob13Cone}.
For the convenience of the reader, we include the simple argument.

First, choose nonzero elements $u_0'',u_0'\in S$ such that $u_0''\ll u_0'\ll u_0$.
Since $S$ is simple and $u_1\neq 0$, we have $u_0'\ll u_0 \leq \infty=\infty u_1$, which allows us to choose $n\geq 1$ such that
$
u_0'\leq nu_1
$.

Applying \axiomO{6} to $u_0''\ll u_0'\leq u_1+\ldots_n+u_1$, we obtain $z_1,\ldots,z_n\in S$ such that
\[
u_0'' \ll z_1+\ldots+z_n, \andSep z_1,\ldots,z_n\ll u_0',u_1.
\]
Since $u_0''$ is nonzero, there is $j\in\{1,\ldots,n\}$ such that $v:=z_j$ is nonzero.
Then $v\ll u_0,u_1$.

Since $S$ is nonelementary, $v$ is not a minimal nonzero element.
Thus, we can choose a nonzero $v'\in S$ with $v'\leq v$ and $v'\neq v$.
Choose a nonzero $v''\in S$ with $v''\ll v'$.
Applying \axiomO{5} to $v''\ll v'\leq v$, we obtain $c\in S$ such that
\[
v''+c\leq v\leq v'+c.
\]
Since $v'\neq v$, we have $c\neq 0$.
Applying the first part of the argument to the nonzero elements $\tilde{u}_0 = v''$ and $\tilde{u}_1= c$, we obtain $w\in S$ such that $0\neq w\ll v'',c$.
Then $w$ has the desired properties.
\end{proof}

%==========================================================================================
\begin{lma}
\label{prp:ThinMonoid}
Let $S$ be a simple \CuSgp{} satisfying \axiomO{5} and \axiomO{6}.
Then $S_{\rm{tb}}$ is a submonoid.
\end{lma}
\begin{proof}
This is clear if $S$ is elementary, since then every element in $S$ way-below another  is compact and therefore $S_{\rm{tb}}=S$;
see \cite[Proposition~5.1.19]{AntPerThi18TensorProdCu}.

We now assume that $S$ is nonelementary.
Let $x,y\in S_\thin$.
To verify that $x+y$ has thin boundary, let $t\in S$ be nonzero.
By \autoref{prp:smallElements}, there is a nonzero element $s$ such that $2s\leq t$. 
This implies
\[
x+y \ll x+s + y+s \leq x+y+t,
\]
as required.
\end{proof}

%==========================================================================================
\begin{lma}
\label{prp:CancelThin}
Let $S$ be a simple, weakly cancellative \CuSgp{} satisfying \axiomO{5}.
Let $x,y,z\in S$ satisfy $x+z\leq y+z$.
Assume that $x,y$ are soft, and that $z$ has thin boundary.
Then $x\leq y$.
\end{lma}
\begin{proof}
If $x=0$ the result is trivial, so we may assume otherwise.

Let $x'\in S$ satisfy $x'\ll x$.
Choose $x''\in S$ such that $x'\ll x''\ll x$. Since $x$ is nonzero and soft, there exists a nonzero $t\in S$ with $x''+t\leq x$. 
Hence,
\[
x'+z \ll x''+(z+t) \leq x+z \leq y+z.
\]
Using weak cancellation, we get $x'\ll y$.

Since this holds for every $x'$ way-below $x$, we get $x\leq y$.
\end{proof}

%==========================================================================================
\begin{lma}
\label{prp:SummandsThin}
Let $S$ be a simple, weakly cancellative \CuSgp{}.
Let $x,y\in S$ such that $x+y$ has thin boundary.
Then $x$ and $y$ have thin boundary.
\end{lma}
\begin{proof}
To show that $x$ has thin boundary, let $t\in S$ be nonzero.
Then 
\[
x+y \ll (x+y)+t = (x+t) + y,
\]
which, by weak cancellation, implies that $x\ll x+t$, as desired.
Analogously, one shows that $y$ has thin boundary.
\end{proof}

%==========================================================================================
\begin{lma}
\label{prp:additionTBPresLL}
Let $S$ be a simple, stably finite \CuSgp{} satisfying \axiomO{5}.
Let $x\in S$ have thin boundary, and let $s,t\in S$ satisfy $s\ll t$.
Assume that $t$ is nonzero and soft.
Then $x+s\ll x+t$.
\end{lma}
\begin{proof}
Choose $t'\in S$ such that $s\ll t'\ll t$.
Since $t$ is nonzero and soft, there exists a nonzero $c\in S$ such that $t'+c\leq t$.
Then
\[
x+s \ll (x+c)+t' \leq x+t,
\]
as desired.
\end{proof}

%==========================================================================================
\begin{dfn}
\label{dfn:complementable}
Let $S$ be a simple, soft \CuSgp.
We say that $x\in S$ is \emph{complementable} if for every $y\in S$ satisfying $x\ll y$ there exists $z\in S$ such that $x+z=y$.
\end{dfn}

%==========================================================================================
The next result implies that elements with thin boundary are complementable;
see \autoref{prp:ThinImplComplementable}.

%==========================================================================================
\begin{prp}
\label{prp:ThinIsComplementable}
Let $S$ be a simple, stably finite \CuSgp{} satisfying \axiomO{5} and \axiomO{6}.
Let $x,y\in S$ satisfy $x\ll y$.
Assume that $x$ has thin boundary and that $y$ is soft.
Then there exists $z\in S$ such that $x+z=y$.
\end{prp}
\begin{proof}
Applying \cite[Proposition~5.1.19]{AntPerThi18TensorProdCu}, the result is clear if $S$ is elementary.
Thus, we may assume that $S$ is nonelementary.
%Then, by \cite[Proposition~5.3.18]{AntPerThi18TensorProdCu}, we obtain that $S_\soft\subseteq S$ is a sub-\CuSgp{}.
The result is also clear if $x=0$, so we may assume that $x\neq 0$.

\textbf{Step~1:}
\emph{We construct an increasing sequence $(y_n)_n$ with supremum $y$ and $x\leq y_0$, and a sequence $(s_n)_n$ of nonzero elements such that 
\[
y_n+s_n\ll y_{n+1}
\]
for every $n\in\NN$.
}

First, let $(\bar{y}_n)_n$ be any $\ll$-increasing sequence in $S$ with supremum $y$. 
%By shifting the index if needed, we may assume that $x\ll \bar{y}_0$.
Set $y_0:=x$.
Since $S$ is simple and stably finite, it follows from \cite[Proposition~5.3.18]{AntPerThi18TensorProdCu} that there exists a soft element $y'$ such that $y_0\ll y'\ll y$. Since $y'$ is nonzero and soft, one can find  a non-zero element $s_0$ such that $y_0+s_0\leq y'$.
%Then $y_0\ll y$, and since $y$ is strongly soft, we obtain a nonzero element $s_0\in S$ such that $y_0+s_0\ll y$.

Using that $y_0+s_0$ and $\bar{y}_1$ are way-below $y$, choose $y_1$ such that
\[
y_0+s_0\ll y_1, \quad \bar{y}_1 \ll y_1, \andSep y_1 \ll y.
\]

Then $y_1\ll y$, and we can apply the previous argument once again to obtain $s_1\neq 0$ such that $y_1+s_1\ll y$.
Using that $y_1+s_1$ and $\bar{y}_2$ are way-below $y$, we obtain $y_2$ such that
\[
y_1+s_1\ll y_2, \quad \bar{y}_2 \ll y_2, \andSep y_2 \ll y.
\]
Continuing this way, we obtain the desired sequences $(y_n)_n$ and $(s_n)_n$.

\textbf{Step~2:}
\emph{We construct a sequence $(r_n)_n$ of nonzero elements such that 
\begin{equation}
\label{prp:ThinIsComplementable:eq1}
2r_{n+1}\ll r_n,s_{n+1}, \andSep y_n+r_n+r_{n+1}\ll y_{n+1}
\end{equation}
for every $n\in\NN$.
}

Applying \autoref{prp:smallElements} for $s_0$, we obtain a nonzero $r_0\in S$ such that $2r_0\ll s_0$.
Then, applying \autoref{prp:smallElements} for $r_0$ and $s_1$, we obtain a nonzero $r_1\in S$ such that $2r_1\ll r_0,s_1$.
Continuing this way, we obtain a sequence $(r_n)_n$ such that $2r_{n+1}\ll r_n,s_{n+1}$ for every $n\in\NN$.

For each $n\in\NN$, we have
\[
y_n + r_n + r_{n+1} \leq y_n + 2r_n \leq y_n + s_n \ll y_{n+1},
\]
which shows that $(r_n)_n$ has the desired properties.

\textbf{Step~3:}
\emph{We construct an $\ll$-increasing sequence $(w_n)_n$ and a sequence $(v_n)_n$ such that
\[
x+r_{n+1}+v_n \leq y_n \leq x+r_n+v_n, \quad
w_n\ll r_{n}+v_{n}, v_{n+1} , \andSep 
y_{n-1} \leq x+w_n
\]
for every $n\geq 1$.
}

To start, using \autoref{prp:additionTBPresLL} at the first step, we have
\[
x + r_2
\ll x + r_1
\leq y_0 + s_0
\leq y_1.
\]
Applying \axiomO{5}, we obtain $v_1\in S$ such that
\[
x + r_2 + v_1  \leq y_1 \leq x+r_1 + v_1.
\]
Using that $y_0\ll y_1$, we can choose $w_1\in S$ such that
\[
y_0 \leq x+w_1, \andSep
w_1\ll r_1+v_1.
\]

Next, let $n\geq 1$, and assume that we have chosen $v_n$ and $w_n$.
Using for the first inequality that $x+r_{n+1}+v_n\leq y_n$ and \eqref{prp:ThinIsComplementable:eq1}, we have
\[
x+r_{n+1}+r_n+v_n
\leq y_{n+1}, \quad
x+r_{n+2} \ll x+r_{n+1}, \andSep
w_n \ll r_n+v_n.
\]
Applying \axiomO{5}, we obtain $v_{n+1}\in S$ such that
\[
x + r_{n+2} + v_{n+1} \leq y_{n+1} \leq x+r_{n+1}+v_{n+1}, \andSep
w_n\ll v_{n+1}.
\]
Using that $y_n\ll y_{n+1}$ and $w_n\ll v_{n+1}\leq  r_{n+1}+v_{n+1}$, we obtain $w_{n+1}\in S$ such that
\[
y_n
\leq x + w_{n+1}, \andSep
w_n\ll w_{n+1} \ll r_{n+1}+v_{n+1}.
\]

Now, the sequence $(w_n)_n$ is increasing, which allows us to set $z:=\sup_n w_n$.
For every $n\geq 1$, we have
\[
x+w_n
\leq x+v_{n+1}
\leq y_{n+2} 
\leq y
\]
and therefore $x+z\leq y$.
Further, for every $n\geq 1$, we have
\[
y_n \leq x+w_{n+1} \leq x+z
\]
and therefore $y\leq x+z$.
This implies $x+z=y$.
\end{proof}

%==========================================================================================
%\begin{rmk}
%\label{rmk:ThinIsComplementable}
%In the proof of \autoref{prp:ThinIsComplementable}, the assumption of stable finiteness is only used to show that the soft element $y$ has the property that for every $y'\in S$ with $y'\ll y$ there exists a nonzero $t\in S$ such that $y'+t\ll y$.
%\end{rmk}

%==========================================================================================
\begin{cor}
\label{prp:ThinImplComplementable}
Let $S$ be a simple, soft, stably finite \CuSgp{} satisfying \axiomO{5} and \axiomO{6}.
Then every element in $S$ with thin boundary is complementable.
\end{cor}

%==========================================================================================
If we additionally assume that $S$ is weakly cancellative, then the converse of \autoref{prp:ThinImplComplementable} also holds:

%==========================================================================================
\begin{thm}
\label{prp:ThinIffComplementable}
Let $S$ be a simple, soft, weakly cancellative \CuSgp{} satisfying \axiomO{5} and \axiomO{6}, and let $x\in S$ satisfy $x\ll\infty$.
Then $x$ has thin boundary if and only if $x$ is complementable.
\end{thm}
\begin{proof}
The forwards implication follows from \autoref{prp:ThinImplComplementable}.
To show the backwards implication, assume that $x$ is complementable.
To verify that $x$ has thin boundary, let $t\in S$ be nonzero.
Choose a nonzero element $t'\in S$ with $t'\ll t$.
Then $x\ll\infty=\infty t'$, which allows us to choose $n\geq 1$ such that $x\leq nt'$.
Choose $t_1,\ldots,t_n\in S$ such that
\[
t'\ll t_1 \ll t_2 \ll \ldots \ll t_n \ll t.
\]
Set $y:=t_1+\ldots+t_n$.
Then $x\leq nt'\ll y$.
Since $x$ is complementable, we obtain $z\in S$ such that $x+z=y$.

Note that
\[
y = t_1+t_2+\ldots+t_{n-1}+t_n
\ll t_2+t_3+\ldots+t_n+t 
\leq y+t,
\]
and therefore
\[
x+z=y\ll y+t = x+z+t.
\]
By weak cancellation, we obtain $x\ll x+t$, as desired.
\end{proof}

%==========================================================================================
\begin{thm}
\label{prp:ThinSummary}
Let $S$ be a simple, soft, weakly cancellative \CuSgp{} satisfying \axiomO{5} and \axiomO{6}.
Then $S_\thin$ is a cancellative monoid.
Further, $x,y\in S_\thin$ satisfy $x\ll y$ if and only if there exists $z\in S_\thin^\times$ with $x+z=y$.
\end{thm}
\begin{proof}
By \autoref{prp:ThinMonoid} and~\ref{prp:CancelThin}, $S_\thin$ is a cancellative monoid.
Let $x,y\in S_\thin$.
If $x\ll y$, then by \autoref{prp:ThinIffComplementable} there exists $z\in S$ such that $x+z=y$.
Since $y$ is not compact, we have $z\neq 0$.
Further, by \autoref{prp:SummandsThin}, we have $z\in S_\thin$.
Conversely, if $z\in S$ is nonzero such that $x+z=y$, then $x\ll x+z=y$ by definition.
\end{proof}

%==========================================================================================
%==========================================================================================
\section{Simple, zero-dimensional Cuntz semigroups}
\label{sec:simple}

%==========================================================================================
In this section, we study countably based, simple, weakly cancellative \CuSgp{s} $S$ that satisfy \axiomO{5} and \axiomO{6} (for example the Cuntz semigroups of separable, simple \ca{s} of stable rank one).
First, we prove a dichotomy:
If~$S$ is zero-dimensional, then $S$ is either algebraic or soft;
see \autoref{prp:DichotomySimpleDim0}.
Conversely, if $S$ is algebraic, then $S$ is automatically zero-dimensional by \autoref{prp:charAlgDim0}.
On the other hand, if $S$ is soft, then $S$ is zero-dimensional if and only if the elements with thin boundary are sup-dense;
see \autoref{prp:CharSimpleSoftDim0}.
We deduce that $S$ is zero dimensional if and only if $S$ is the retract of a simple, algebraic \CuSgp{}; see \autoref{prp:CharSimpleDim0ByRetract}.

This should be compared with \autoref{prp:simple_Zstable}, where we showed that a separable, simple, $\mathcal{Z}$-stable \ca{} has zero-dimensional Cuntz semigroup if and only if $A$ has real rank zero or $A$ is stably projectionless.

%==========================================================================================
\begin{lma}
\label{prp:DichotomySimpleDim0}
Let $S$ be a simple, weakly cancellative \CuSgp{} satisfying \axiomO{5}.
Assume that $\dim(S)=0$ and $S\neq\{0\}$.
Then, $S$ is either algebraic or soft.
\end{lma}
\begin{proof}
Assume that $S$ is not soft.
Then there exists a nonzero compact element in $S$, which by \autoref{prp:AlgebraicIdeal} implies that $S$ is algebraic.
\end{proof}

%=============================================
\begin{lma}
\label{prp:DenseThinImplRiesz}
Let $S$ be a simple, soft, weakly cancellative \CuSgp{} satisfying \axiomO{5} and \axiomO{6}.
Assume that $S_\thin$ is sup-dense.
Let $x,y,z\in S$ satisfy $x\ll y+z$, and assume that $x$ has thin boundary.
Then there exist $v,w\in S_\thin$ such that
\[
x=v+w, \quad v\ll y, \andSep w\ll z.
\]
%Then $S_\thin$ satisfies Riesz decomposition.
\end{lma}
\begin{proof}
We may assume that $z$ is nonzero, since otherwise $v=x$ and $w=0$ trivially satisfy the required conditions. 
Choose $z'\in S$ such that
\[
x\ll y+z', \andSep z'\ll z.
\]
Using that $z$ is nonzero and soft, we obtain a nonzero $t\in S$ such that $z'+t\ll z$.
Since $x$ has thin boundary, we have $x\ll x+t$, which allows us to choose $x'\in S$ such that
\[
x'\ll x\ll x'+t.
\]
Since $S_\thin$ is sup-dense, we may assume that $x'$ has thin boundary.

Applying \axiomO{6} to $x'\ll x\ll y+z'$, we obtain $e,f\in S$ such that
\[
x'\ll e+f, \quad e\ll x,y, \andSep f\ll x,z'.
\]
Since $S_\thin$ is sup-dense, we may assume that $e$ has thin boundary.

By \autoref{prp:ThinImplComplementable}, $e$ is complementable.
Thus, we obtain $c\in S$ such that $e+c=x$.
Then
\[
e+c = x \ll x'+t \leq e+f+t.
\]
By weak cancellation, we get $c\ll f+t$ and therefore
\[
c \ll f+t \leq z'+t \ll z.
\]
By \autoref{prp:SummandsThin}, $e$ and $c$ have thin boundary.
Hence, $v:=e$ and $w:=c$ have the desired properties.
\end{proof}

%=============================================
\begin{prp}
\label{prp:DenseThinImplDim0}
Let $S$ be a simple, soft, weakly cancellative \CuSgp{} satisfying \axiomO{5} and \axiomO{6}.
Assume that $S_\thin$ is sup-dense.
Then $S_\thin$ is a simple, cancellative refinement monoid and $\dim(S)=0$.
\end{prp}
\begin{proof}
By \autoref{prp:ThinSummary}, $S_\thin$ is a cancellative monoid such that $x,y\in S_\thin$ satisfy $x\ll y$ if and only if there exists $z\in S_\thin^\times$ with $x+z=y$.
This implies that $x,y\in S_\thin$ satisfy $x\leq_\alg y$ if and only if $x=y$ or $x\ll y$.

It follows from \autoref{prp:DenseThinImplRiesz} that $S_\thin$ satisfies the Riesz decomposition property for the pre-order induced by $\ll$.
Hence, $\dim(S)=0$ by \autoref{prp:denseRiesz}.

Since $S_\thin$ is a cancellative monoid, to show that it is a refinement monoid it suffices to show that it satisfies the Riesz decomposition property for the algebraic partial order $\leq_\alg$.
Let $x,y,z\in S_\thin$ satisfy $x\leq_\alg y+z$.
We need to find $y',z'\in S_\thin$ such that $x=y'+z'$, $y'\leq_\alg y$ and $z'\leq_\alg z$.
We either have $x=y+z$ or $x\ll y+z$.
In the first case, $y':=y$ and $z':=z$ have the desired properties.
In the second case, we apply \autoref{prp:DenseThinImplRiesz} to obtain $y',z'\in S_\thin$ such that $x=y'+z'$, $y'\ll y$ and $z'\ll z$.
Then $y'\leq_\alg y$ and $z'\leq_\alg z$, which shows that $y'$ and $z'$ have the desired properties.
Using that $S$ is simple, it easily follows that $S_\thin$ is a simple monoid.
\end{proof}

%=============================================
\begin{exa}
\label{exa:DenseThinNotDim0}
Let $Z$ be the Cuntz semigroup of the Jiang-Su algebra $\mathcal{Z}$.
Then every element of $Z$ has thin boundary, yet $Z$ is neither algebraic nor soft, and therefore $Z$ is not zero-dimensional.
(We have $\dim(Z)=1$ by \autoref{exa:Z}.)

This shows that \autoref{prp:DenseThinImplDim0} does not hold without assuming that $S$ is soft.
\end{exa}

%==========================================================================================
Next, we prove the converse of  \autoref{prp:DenseThinImplDim0}:
Zero-dimensionality implies that $S_\thin$ is sup-dense.
We start with a crucial technical result.
% that we will use to study the structure of zero-dimensional \CuSgp{s}.

%==========================================================================================
\begin{lma}
\label{prp:ApproxIfDominated}
Let $S$ be a weakly cancellative \CuSgp{} satisfying \axiomO{5}, and let $x',x'',x,e,t\in S$ satisfy
\[
x'\ll x'', \andSep
x''+t \leq x\leq e\ll e+t.
\]
Assume that $\dim (S)=0$.
Then there exists $y$ such that
\[
x' \ll y \ll x, \andSep
y \ll y+t.
\]
\end{lma}
\begin{proof}
Applying \axiomO{5} to $x'\ll x''\leq e$, we obtain $c\in S$ such that
\[
x'+c \leq e \leq x''+c.
\]
Then
\[
e \ll e+t \leq x''+c+t.
\]
Using that $\dim(S)=0$, we obtain $u,v\in S$ such that
\[
u \ll x'', \quad
v \ll c+t, \andSep
e\ll u+v \ll e+t.
\]
Then
\[
x'+c \leq e \ll u+v \leq u+c+t,
\]
Using weak cancellation, we get $x'\ll u+t$.
Further, we have
\[
u+v \ll e+t \leq u+v+t
\]
and therefore $u\ll u+t$ by weak cancellation.

Choose $t'\in S$ such that
\[
t'\ll t, \quad
x' \ll u+t', \andSep
u \ll u+t'.
\]
Set $y:=u+t'$.
Then
\[
x' \ll u+t' = y, \andSep
y = u+t' \ll x''+t \leq x.
\]
Using that $u\ll u+t'$ and $t'\ll t$, we get
\[
y = u+t' \ll u+t'+t = y+t,
\]
which shows that $y$ has the desired properties.
\end{proof}

%==========================================================================================
\begin{lma}
\label{prp:ApproxDensThin}
Let $S$ be a countably based, simple, soft, weakly cancellative \CuSgp{} satisfying \axiomO{5} and \axiomO{6}.

Assume that for every $x',x,t\in S$ satisfying $x'\ll x$ and $t\neq 0$ there exists $y\in S$ such that
\[
x' \ll y \ll x, \quad y\ll y+t.
\]

Then for every $x',x\in S$ satisfying $x'\ll x$ there exists $y\in S$ with thin boundary such that $x'\ll y\ll x$.
\end{lma}
\begin{proof}
Using that $S$ is countably based, we can choose a sequence $(t_n)_{n\in\NN}$ of nonzero elements such that for every nonzero $t\in S$ there exists $n$ with $t_n\leq t$.

To prove the statement, let $x',x\in S$ satisfy $x'\ll x$.
By assumption, we can choose $y_0\in S$ with $x'\ll y_0\ll x$ and $y_0\ll y_0+t_0$.
Choose $y_0'$ such that
\[
x'\ll y_0'\ll y_0 \ll x, \quad y_0\ll y_0'+t_0.
\]

Next, applying the assumption for $y_0',y_0,t_1$, we obtain $y_1$ such that $y_0'\ll y_1\ll y_0$ and $y_1\ll y_1+t_1$.
Then choose $y_1'$ such that
\[
y_0'\ll y_1'\ll y_1 \ll y_0, \quad y_1\ll y_1'+t_1.
\]

Inductively, choose $y_n'$ and $y_n$ such that
\[
x'\ll y_0' \ll \ldots \ll y_n'\ll y_n \ll \ldots \ll y_0 \ll x, \quad 
y_n\ll y_n'+t_n.
\]

Set $y:=\sup_n y_n'$.
Then $x'\ll y_0' \leq y \leq y_0\ll x$.
To show that $y$ has thin boundary, let $t\in S$ be nonzero.
By choice of $(t_n)_n$, there exists $n$ such that $t_n\leq t$.
Then
\[
y \leq y_n \ll y_n'+t_n \leq y+t_n \leq y+t,
\]
as desired.
\end{proof}

%=============================================
\begin{prp}
\label{prp:Dim0ImplThinDense}
Let $S$ be a countably based, simple, soft, weakly cancellative \CuSgp{} satisfying \axiomO{5} and \axiomO{6}.
Assume that $\dim(S)=0$.
Then $S_\thin$ is sup-dense, that is, the elements with thin boundary form a basis.
\end{prp}
\begin{proof}
%The proof is similar to the proof that complementable elements have thin boundary in \autoref{prp:ThinIffComplementable}.
We verify the assumption of \autoref{prp:ApproxDensThin}, which then proves the statement.
Let $x',x,t\in S$ satisfy $x'\ll x$ and $t\neq 0$.
We need to find $y\in S$ such that
\[
x' \ll y \ll x, \quad y\ll y+t.
\]
If $x'=0$, then set $y:=0$.
Thus, we may assume from now on that $x'$ is nonzero.

Choose $x'',u\in S$ such that
\[
x'\ll x''\ll u \ll x.
\]
Since $u$ is nonzero and soft, we obtain a nonzero element $s\in S$ such that $x''+s\ll u$.
By \autoref{prp:smallElements}, there exists a nonzero $r\in S$ with $r\leq s,t$.

Choose a nonzero $r'\in S$ such that $r'\ll r$.
Then $u\ll\infty=\infty r'$, which allows us to choose $n\geq 1$ such that $u\leq nr'$.
Choose $r_1,\ldots,r_n\in S$ such that
\[
r'\ll r_1 \ll r_2 \ll \ldots \ll r_n \ll r.
\]
Set $e:=r_1+\ldots+r_n$.
As in the proof of \autoref{prp:ThinIffComplementable}, we obtain $e\ll e+r$, and consequently $e\ll e+t$.
Further, we have
\[
x' \ll x'', \quad
x''+r \leq x''+s \ll u \leq nr' \leq e \ll e+t.
\]
Applying \autoref{prp:ApproxIfDominated}, we obtain $y\in S$ such that
\[
x' \ll y \ll u, \andSep y\ll y+t,
\]
Now $y$ has the desired properties.
\end{proof}

%==========================================================================================
\begin{thm}
\label{prp:CharSimpleSoftDim0}
Let $S$ be a countably based, simple, soft, weakly cancellative \CuSgp{} satisfying \axiomO{5} and \axiomO{6}.
Then, the following conditions are equivalent:
\begin{enumerate}
\item
$\dim(S)=0$;
\item
the elements with thin boundary are sup-dense;
\item
there exists a countably based, simple, algebraic, weakly cancellative \CuSgp{} $T$ satisfying \axiomO{5} and \axiomO{6} such that $S\cong T_\soft$.
%\item
%there exists a (countable, conical, cancellative, non-atomic, simple) refinement monoid $M$ such that $S^\times\cong\Lambda_\sigma(M^\times)$.
\end{enumerate}
\end{thm}
\begin{proof}
By \autoref{prp:Dim0ImplThinDense}, (1) implies~(2).
Conversely, (2) implies~(1) by \autoref{prp:DenseThinImplDim0}.
To show that~(3) implies~(1), let $T$ be as in~(3) such that $S\cong T_\soft$.
Using \autoref{prp:DimSoftPart} at the second step, and using \autoref{prp:charAlgDim0} at the last step, we get 
\[
\dim(S) 
= \dim(T_\soft)
\leq\dim(T)
=0.
\]

Finally, assuming~(2) let us verify~(3).
By \autoref{prp:DenseThinImplDim0}, $S_\thin$ is a simple, cancellative refinement monoid.
Using that $S$ is countably based and that $S_\thin$ is sup-dense, we can choose a countable subset $M_0\subseteq S_\thin$ that is sup-dense.

By successively adding elements to $M_0$ we can construct a countable refinement submonoid $M\subseteq S_\thin$ such that the algebraic order on $M$ agrees with the restriction of the algebraic order on $S_\thin$ to $M$, that is, $(M,\leq_\alg)\to(S_\thin,\leq_\alg)$ is an order-embedding.
Set $T:=\Cu(M,\leq_\alg)$, the sequential round ideal completion of $M$ with respect to the algebraic partial order;
see \cite[Section~5.5]{AntPerThi18TensorProdCu}.
Then $T$ is a countably based, algebraic \CuSgp.
Using that $M$ is a cancellative monoid that is algebraically ordered and that satisfies the Riesz decomposition property, it follows from \cite[Proposition~5.5.8]{AntPerThi18TensorProdCu} that $T$ is weakly cancelaltive and satisfies \axiomO{5} and \axiomO{6}.
Using that $M$ is a simple monoid, it follows that $T$ is simple.

Recall that a subset $I\subseteq M$ is an interval if $I$ is downward hereditary and upward directed.
Since $M$ is countable, we can identify $T$ with the set of intervals in $M$, ordered by inclusion.
The compact elements in $T$ are precisely the intervals $\{y\in M:y\leq_\alg x\}$ for $x\in M$.
Thus, the nonzero soft elements in $T$ are precisely the intervals that do not contain a largest element.
Using that every upward directed set in a countably based \CuSgp{} has a supremum, we can define $\alpha\colon T_\soft\to S$ by
\[
\alpha(I) := \sup I,
\]
for every (soft) interval $I\subseteq M$.
It is now straightforward to verify that $\alpha$ is an isomorphism.
\end{proof}

%==========================================================================================
\begin{rmk}
There is no canonical choice for the algebraic \CuSgp{} $T$ in \autoref{prp:CharSimpleSoftDim0}(3).
Take for example $S=[0,\infty]$.
For every supernatural number~$q$ satisfying $q=q^2\neq 1$, we consider the UHF-algebra $M_q$ of infinite type, and set $R_q:=\Cu(M_q)$;
see \cite[Section~7.4]{AntPerThi18TensorProdCu}.
Then $R_q$ is a countably based, simple, algebraic, weakly cancellative \CuSgp{} satisfying \axiomO{5} and \axiomO{6}, and $(R_q)_\soft\cong[0,\infty]$.

Given a countably based, simple, soft, weakly cancellative \CuSgp{} $S$ satisfying \axiomO{5} and \axiomO{6}, we can consider $T:=\Cu(S,\leq_\alg)$, which is a simple, algebraic, weakly cancellative \CuSgp{} satisfying \axiomO{5} and \axiomO{6} such that $S\cong T_\soft$.
However, $T$ is not countably based in general since every basis of $T$ contains all compact elements of $T$ and so has at least the cardinality of $S_\thin$.
\end{rmk}

%==========================================================================================
Recall the notion of a retract from \autoref{dfn:retract}.

%==========================================================================================
\begin{thm}
\label{prp:CharSimpleDim0ByRetract}
Let $S$ be a countably based, simple, weakly cancellative \CuSgp{} satisfying \axiomO{5} and \axiomO{6}.
Then $S$ is zero-dimensional if and only if $S$ is a retract of a countably based, simple, algebraic, weakly cancellative \CuSgp{} satisfying \axiomO{5} and \axiomO{6}.
\end{thm}
\begin{proof}
By \autoref{prp:DichotomySimpleDim0}, $S$ is either algebraic or soft.
In the first case, we consider $S$ as a retract of itself.
In the second case, the result follows from \autoref{prp:CharSimpleSoftDim0} and \autoref{prp:retractSimpleSoft}.
\end{proof}

%==========================================================================================
\begin{qst}
\label{qst:Retract}
Is every zero-dimensional, weakly cancellative \CuSgp{} satisfying \axiomO{5} a retract of a weakly cancellative, algebraic \CuSgp{} satisfying \axiomO{5} and \axiomO{6}?
\end{qst}

%==========================================================================================
Recall that a partially ordered set $M$ has the \emph{Riesz interpolation property} if for all $x_0,x_1,y_0,y_1\in M$ satisfying $x_j\leq y_k$ for all $j,k\in\{0,1\}$, there exists $z\in M$ such that $x_j\leq z\leq y_k$ for all $j,k\in\{0,1\}$.
By \cite[Theorem~3.5]{AntPerRobThi18arX:CuntzSR1}, Cuntz semigroups of stable rank one \ca{s} have the Riesz interpolation property.

Recall that a \CuSgp{} $S$ is said to be \emph{almost divisible} if for all $n\in\NN$ and $x',x\in S$ satisfying $x'\ll x$ there exists $y\in S$ such that $ny\leq x$ and $x'\leq(n+1)y$;
see \cite[Definition~7.3.4]{AntPerThi18TensorProdCu}.

%==========================================================================================
\begin{lma}
\label{prp:RetractInterpol}
Let $S$ be a retract of a \CuSgp{} $T$.
Then, if $T$ is almost divisible, so is $S$.
Further, if $T$ has the Riesz interpolation property, then so does~$S$.
\end{lma}
\begin{proof}
Let $\iota\colon S\to T$ be a \CuMor{}, and let $\sigma\colon T\to S$ be a generalized \CuMor{} with $\sigma\circ\iota=\id_S$.

First, assume that $T$ has the Riesz interpolation property. 
%To verify that $S$ has the Riesz interpolation property, 
Let $x_0,x_1,y_0,y_1\in S$ satisfy $x_j\leq y_k$ for all $j,k\in\{0,1\}$.
Then $\iota(x_j)\leq\iota(y_k)$ in $T$ for all $j,k\in\{0,1\}$.
By assumption, there is $z\in T$ such that $\iota(x_j)\leq z\leq\iota(y_k)$ and thus $x_j\leq\sigma(z)\leq y_k$ for all $j,k\in\{0,1\}$.
Thus, $\sigma(z)$ has the desired properties.

Next, assume that $T$ is almost divisible.
Let $n\in\NN$ and let $x',x\in S$ satisfy $x'\ll x$.
Then $\iota(x')\ll\iota(x)$ in $T$.
By assumption, there exists $y\in T$ such that $ny\leq \iota(x)$ and $\iota(x')\leq(n+1)y$.
Then $n\sigma(y)\leq x$ and $x'\leq(n+1)\sigma(y)$.
\end{proof}

%==========================================================================================
\begin{prp}
\label{prp:Dim0ImplInterpol}
Let $S$ be a zero-dimensional, countably based, simple, weakly cancellative, nonelementary  \CuSgp{} satisfying \axiomO{5}. 
Then $S$ satisfies the Riesz interpolation property and is almost divisible.
\end{prp}
\begin{proof}
By \autoref{prp:CharSimpleDim0ByRetract}, there exists a countably based, simple, algebraic, weakly cancellative \CuSgp{} $T$ satisfying \axiomO{5} and \axiomO{6} such that $S$ is a retract of $T$.
Then $T_c$ is a simple, cancellative refinement monoid, and therefore $T_c$ has the Riesz interpolation property.
Hence, $T$ has the Riesz interpolation property by \cite[Proposition~5.5.8(3)]{AntPerThi18TensorProdCu}.

Since $S$ is nonelementary, it follows from \cite[Theorem~6.7]{AraGooPerSil10NonSimplePI} that $T_c$ is weakly divisible, that is, for every $x\in T_c$ there exist $y,z\in T_c$ such that $x=2y+3z$.
This implies that $T$ is almost divisible.

Now the result follows from \autoref{prp:RetractInterpol}.
\end{proof}

%==========================================================================================
\begin{rmk}
\label{rmk:Converse}
One might wonder to what extent the converse of \autoref{prp:Dim0ImplInterpol} holds:
Given a countably based, simple, \emph{soft}, weakly cancellative, almost divisible \CuSgp{} $S$ that satisfies \axiomO{5}, \axiomO{6} and the Riesz interpolation property, is~$S$ zero-dimensional?

We thank the referee for pointing out that without softness, this question has a negative answer.
Indeed, the Cuntz semigroup of the Jiang-Su algebra is countably based, simple, weakly cancellative, almost divisible, and it satisfies \axiomO{5}, \axiomO{6} and the Riesz interpolation property -- nevertheless, it is not zero-dimensional by \autoref{exa:Z}.
\end{rmk}

%==========================================================================================
\begin{qst}
\label{qst:Interpol}
Let $S$ be a zero-dimensional, weakly cancellative \CuSgp{} satisfying \axiomO{5}.
Does $S$ have the Riesz interpolation property?
Assuming also that~$S$ has no elementary quotients, is $S$ almost divisible?
\end{qst}

%==========================================================================================
Note that a positive answer to \autoref{qst:Retract} entails a positive answer to \autoref{qst:Interpol}.

%==========================================================================================
%\bibliographystyle{../../aomalphaMyShort}
%\bibliography{../../References}

\providecommand{\etalchar}[1]{$^{#1}$}

\end{document}